\documentclass[11pt,reqno]{amsart}
\usepackage{amssymb,amsmath,amsfonts,amsthm,enumerate,stmaryrd}
\usepackage[left=.75 in, right=.75 in,top=.75 in, bottom=.75 in]{geometry}
\usepackage{nameref,hyperref,cleveref}

\providecommand{\abs}[1]{\left\vert#1\right\vert}
\providecommand{\norm}[1]{\left\Vert#1\right\Vert}

\def\nab{\nabla}
\def\dt{\partial_t}

\def\ls{\lesssim}
\def\p{\partial}

\def\C{\mathbb{C}}

\def\N{\mathbb{N}}

\def\R{\mathbb{R}}

\def\T{\mathbb{T}}
\def\Z{\mathbb{Z}}

\def\st{\;\vert\;}

\DeclareMathOperator{\diverge}{div}

\def\Z{\mathbb{Z}}

\def\R{\mathbb{R}}
\def\N{\mathbb{N}}
\def\C{\mathbb{C}}

\def\T{\mathbb{T}}

\def\Div{\text{div}}

\def\half{\frac{1}{2}}
\def\wo{\backslash}
\def\ben{\begin{enumerate}}
\def\een{\end{enumerate}}
\def\bit{\begin{itemize}}
\def\eit{\end{itemize}}

\def\Hbot{{\mathbin{{}_0{H}^1}}}
\def\Hhel{{\mathbin{{}_0{H}^1_{\text{sol}}}}}
\def\Hxihel{{H^1_{\xi,s}}}
\def\Hxi{{H^1_\xi}}

\def\lesim{\lesssim}
\def\gesim{\gtrsim} 
\def\eell{\lambda}

\title[Energy decay and fractional boundary operators]{Decay of solutions to the linearized free surface Navier-Stokes equations with fractional boundary operators}

\author{Ian Tice}
\address{
Department of Mathematical Sciences\\
Carnegie Mellon University\\
Pittsburgh, PA 15213, USA
}
\email[I. Tice]{iantice@andrew.cmu.edu}
\thanks{I. Tice was supported by a Simons Foundation Grant (\#401468) and an NSF CAREER Grant (DMS \#1653161). }

\author{Samuel Zbarsky}
\address{
Department of Mathematics\\
Princeton University\\
Princeton, NJ 08544, USA
}
\email[S. Zbarsky]{szbarsky@math.princeton.edu}
\thanks{S. Zbarsky was supported by a National Science Foundation Graduate Research Fellowship. }


\newtheorem{lem}{Lemma}[section]

\newtheorem{thm}[lem]{Theorem}
\newtheorem{remark}[lem]{Remark}

\theoremstyle{definition}
\numberwithin{equation}{section} 

\begin{document}

\begin{abstract}
In this paper we consider a slab of viscous incompressible fluid bounded above by a free boundary, bounded below by a flat rigid interface, and acted on by gravity.  The unique equilibrium is a flat slab of quiescent fluid.  It is well-known that equilibria are asymptotically stable but that the rate of decay to equilibrium depends heavily on whether or not surface tension forces are accounted for at the free interface.  The aim of the paper is to better understand the decay rate by studying a generalization of the linearized dynamics in which the surface tension operator is replaced by a more general fractional-order differential operator, which allows us to continuously interpolate between the case without surface tension and the case with surface tension.  We study the decay of the linearized problem in terms of the choice of the generalized operator and in terms of the horizontal cross-section.  In the case of a periodic cross-section we identify a critical order of the differential operator at which the decay rate transitions from almost exponential to exponential.
\end{abstract}

\maketitle

%\pagebreak
%\tableofcontents
%\pagebreak

%\pagebreak
%\tableofcontents
%\pagebreak

%%%%%%%%%%%%%%%%%%%%%%%%%%%%%%%%%%%%%%%%%%%%%%%%%
%%%%%%%%%%%%%%%%%%%%%%%%%%%%%%%%%%%%%%%%%%%%%%%%%
%%%%%%%%%%%%%%%%%%%%%%%%%%%%%%%%%%%%%%%%%%%%%%%%%
\section{Introduction }
%%%%%%%%%%%%%%%%%%%%%%%%%%%%%%%%%%%%%%%%%%%%%%%%% 
%%%%%%%%%%%%%%%%%%%%%%%%%%%%%%%%%%%%%%%%%%%%%%%%%
%%%%%%%%%%%%%%%%%%%%%%%%%%%%%%%%%%%%%%%%%%%%%%%%%

%%%%%%%%%%%%%%%%%%%%%%%%%%%%%%%%%%%%%%%%%%%%%%%%%
%%%%%%%%%%%%%%%%%%%%%%%%%%%%%%%%%%%%%%%%%%%%%%%%%
\subsection{Free boundary Navier-Stokes equations}
%%%%%%%%%%%%%%%%%%%%%%%%%%%%%%%%%%%%%%%%%%%%%%%%% 
%%%%%%%%%%%%%%%%%%%%%%%%%%%%%%%%%%%%%%%%%%%%%%%%%

Consider the evolution of a layer of viscous incompressible fluid subject to a constant gravitational force.  The physically relevant dimensions for the evolution are either $3$ or $2$, but we will actually consider arbitrary dimension $N \ge 2$, as the results of this paper ultimately do not depend on the dimension.  The fluid is bounded below by a flat hyper-plane, corresponding to an interface with a rigid solid, and above by a free surface that evolves with the fluid, corresponding to the interface with another fluid.  We assume that the fluid lying above the moving interface is trivial in the sense that it is of constant pressure (a vacuum, for example).  We will consider either the case in which the horizontal extent of the fluid is infinite, in which case its cross section is 
\begin{equation}
\Sigma = \R^{N-1}, 
\end{equation}
or else the case in which the fluid is horizontally periodic, in which case we assume its cross section is the $(N-1)$-torus 
\begin{equation}
\Sigma = \T^{N-1}. 
\end{equation}
More general periodicity could be enforced by allowing the periodicity length to vary in each coordinate direction, but this does not have any serious impact on the results of the paper, so we have chosen the simplest setting of a standard torus.

Let us now state the equations of motion governing the evolution.  Throughout the paper we will write points $z \in \Sigma \times \R$ as $z = (x,y)$ with $x \in \Sigma$ and $y \in \R$.  In other words, $x$ is the horizontal variable and $y$ is the vertical variable.  We continue to write $\p_i = \p_{z_i}$ for partial derivatives, with the understanding that $z_i = x_i$ for $1 \le i \le N-1$ and $z_N = y$.  

We assume that the moving upper boundary of the fluid is given by the graph of an unknown function $\eta: \Sigma \times [0,\infty) \to (0,\infty)$, which means that the moving fluid domain is modeled by the open $N-$dimensional set
\begin{equation}
\Omega(t) = \{ z = (x,y) \in \Sigma \times \R \st 0 < y < \eta(x,t)\}.
\end{equation}
Then the moving upper surface is 
\begin{equation}
 \Sigma(t) = \{ z =(x,y) \in \Sigma \times \R \st  y = \eta(x,t)\}.
\end{equation}

For each $t\ge 0$ the fluid is determined by its velocity and pressure functions $(u,\bar{p}) :\Omega(t) \to \R^N \times \R$.  Then the unknowns $(u, \bar{p}, \eta)$ must satisfy the  incompressible
Navier-Stokes equations in $\Omega(t)$ for $t>0$: 
\begin{equation}\label{ns_euler_prim}
\begin{cases}
\partial_t u + u \cdot \nabla u + \nabla \bar{p} =   \Delta u - g e_N & \text{in }
\Omega(t) \\ 
\diverge{u}=0 & \text{in }\Omega(t) \\ 
\partial_t \eta = u_N - \sum_{j=1}^{N-1} u_j \partial_{j}\eta  & 
\text{on } \Sigma(t) \\ 
(\bar{p} I -    \mathbb{D}u ) \nu = (P_{ext}   - \sigma \mathcal{H}(\eta))\nu & \text{on } \Sigma(t) \\ 
u = 0 & \text{on } \{ y = 0\}.
\end{cases}
\end{equation}
Here $e_N = (0,\dotsc,0,1) \in \R^N$ is the vertical unit vector,  $g >0$ is the strength of the gravitational force,
\begin{equation}
(\mathbb{D} u)_{ij}  = \partial_i u_j + \partial_j u_i 
\end{equation}
is the symmetrized gradient of $u$,  $\nu$ is the outward-pointing unit normal vector on the moving surface $\Sigma(t)$,  $I$ is the $N \times N$ identity matrix, $P_{ext} \in \R$ is the constant pressure above the fluid, $\sigma \ge 0$ is the surface tension coefficient, and 
\begin{equation}\label{H_def}
 \mathcal{H}(\eta) = \diverge\left( \frac{\nab \eta}{\sqrt{1+\abs{\nab \eta}^2}}\right)
\end{equation}
is (minus) twice the mean curvature of the free interface.  The first two equations in \eqref{ns_euler_prim} are the standard incompressible Navier-Stokes equations with unit density and viscosity.  The third is the kinematic transport equation for $\eta$, the fourth is the balance of stress at the interface, and the fifth is the no-slip boundary condition at the bottom.  The problem is augmented with initial data $\eta_0 : \Sigma \to (0,\infty)$, which determines the initial domain $\Omega_0$, as well as an initial velocity field $u_0 : \Omega_0 \to \R^N$.  Note that the assumption  $\eta_0 > 0$ on $\Sigma$ means that $\Omega_0$ is a well-defined open, connected set.   

It is convenient to rewrite \eqref{ns_euler_prim} in an equivalent form.  To do so we introduce an equilibrium height $\ell >0$ and redefine the pressure via 
\begin{equation}
 p(z,t) = \bar{p}(z,t) - P_{ext} + g(y - \ell).
\end{equation}
We then introduce the variation from equilibrium height as 
\begin{equation}
 h(x,t) = \eta(x,t) - \ell.
\end{equation}
In terms of these new unknowns, the system may be rewritten as
\begin{equation}\label{ns_euler}
\begin{cases}
\partial_t u + u \cdot \nabla u + \nabla p =   \Delta u  & \text{in }
\Omega(t) \\ 
\diverge{u}=0 & \text{in }\Omega(t) \\ 
\partial_t h = u_N - \sum_{j=1}^{N-1} u_j \partial_{j} h  & 
\text{on } \Sigma(t) \\ 
(p I -    \mathbb{D}u ) \nu = (gh   - \sigma \mathcal{H}(h))\nu & \text{on } \Sigma(t) \\ 
u = 0 & \text{on } \{ y = 0\}.
\end{cases}
\end{equation}

%%%%%%%%%%%%%%%%%%%%%%%%%%%%%%%%%%%%%%%%%%%%%%%%%
%%%%%%%%%%%%%%%%%%%%%%%%%%%%%%%%%%%%%%%%%%%%%%%%%
\subsection{Equilibria, stability, and linearization}
%%%%%%%%%%%%%%%%%%%%%%%%%%%%%%%%%%%%%%%%%%%%%%%%% 
%%%%%%%%%%%%%%%%%%%%%%%%%%%%%%%%%%%%%%%%%%%%%%%%%
 
Associated to sufficiently regular solutions of \eqref{ns_euler} are physical energy and dissipation functionals:  the energy is 
\begin{equation}\label{physenergy}
E(t)=\int_{\Omega(t)} \half |u(z,t)|^2 dz +\half \int_\Sigma g \abs{h(x,t)}^2 dx +\sigma \int_\Sigma (\sqrt{1+|\nabla h(x,t)|^2}-1) dx,
\end{equation}
and the dissipation is 
\begin{equation}
D(t) =  \half\int_{\Omega(t)} |{\mathbb D}u(z,t)|^2 dz.
\end{equation}
The first term in the energy is the total kinetic energy of the fluid, the second is the gravitational potential energy stored in the fluid, and the third is the surface energy generated by deviation from a flat interface.  The dissipation functional measures the extent to which the fluid flow deviates from a rigid motion.  The energy and dissipation are related through the energy-dissipation equation:
\begin{equation}\label{ed_phsx}
 \frac{d}{dt} E(t) + D(t) =0,
\end{equation}
as can be seen by taking the dot product of the first equation in \eqref{ns_euler} with $u$, integrating by parts, and using the other equations in \eqref{ns_euler}.  The identity \eqref{ed_phsx} and the non-negativity of $D$ show that the energy is non-increasing in time and that any decrease in the energy is accounted for by the dissipation functional.

The energy-dissipation equation \eqref{ed_phsx} reveals that any equilibrium (time-independent) solution must satisfy $D=0$, and so the no-slip boundary condition and Korn's inequality, Theorem \ref{korn}, imply that $u=0$.  In turn, this implies that $p=0$ and $h=0$.  We deduce from this that the only equilibrium solution corresponds to a quiescent fluid in the flat slab $\Sigma \times (0,\ell)$.  

Clearly $u =0$, $p=0$, $h=0$ corresponds to a global minimizer of the energy functional $E$, so we formally expect the equilibrium solution to be stable.  This is indeed the case, and in fact the equilibrium is asymptotically stable, though the rate of equilibration is highly sensitive to the form of $\Sigma$ and the value of $\sigma$.  We now survey some of the known results.  General data local well-posedness of solutions without surface tension ($\sigma=0$) was first proved by Beale \cite{Beale1981}.  In \cite{Beale1983}, Beale showed global existence with surface tension ($\sigma >0$) and $\Sigma=\R^{2}$ for small initial data.  Algebraic decay in this case, assuming the initial data belong to $L^1$, was proved by Beale-Nishida \cite{BealeNishida1985}. In \cite{NishidaTTYYH2004}, Nishida-Teramoto-Yoshihara showed that with surface tension, $\Sigma=\T^{2}$, and with small initial data, the energy decays exponentially.  In \cite{Hataya2009}, Hataya showed that without surface tension when $\Sigma = \T^2$, if the data are small in a certain Sobolev space, then the energy decays algebraically: $E(t)\lesim (1+t)^{-2}$.  In \cite{GuoTice2013a2}, Guo-Tice showed algebraic decay without surface tension when $\Sigma=\R^{2}$ and with small initial data.  In \cite{GuoTice2013}, Guo-Tice showed that when $\Sigma=\T^{2}$ without surface tension, we get almost exponential decay of the total energy for small initial data, namely if $M\in\N$ with $M\ge 5$ and $||u_0||_{H^{2M}}^2+||\eta_0||_{H^{2M+\half}}^2$ is sufficiently small, then
\begin{equation}
E(t)\lesim (1+t)^{-4M}. 
\end{equation}
We summarize the known decay rates in the following table.
\begin{displaymath}
\begin{array}{| c | c | c | }
\hline
& 
\Sigma = \R^2 & 
\Sigma = \T^2   \\ \hline
%%%%%%%%%%%%%%%%%
 \sigma > 0 &
\textnormal{Algebraic decay} &
\textnormal{Exponential decay}    \\ \hline
%%%%%%%%%%%%%%%%%
\sigma =0 &
\textnormal{Algebraic decay} &
\textnormal{Almost exponential decay}     \\ \hline
%%%%%%%%%%%%%%%%%
\end{array}
\end{displaymath}

We can get a rough understanding of the role that surface tension plays in determining the decay rate by examining the linearization of \eqref{ns_euler}.  The linearized problem is posed in the fixed equilibrium domain 
\begin{equation}
 \Omega = \Sigma \times (0,\ell),
\end{equation}
with the linearized unknowns $(u,p,\eta)$ evolving according to the linearized equations:
\begin{equation}\label{ns_linear}
\begin{cases}
\partial_t u +  \nabla p =   \Delta u  & \text{in }
\Omega \\ 
\diverge{u}=0 & \text{in }\Omega \\ 
\partial_t h = u_N   & 
\text{on } \{y = \ell\} \\ 
(p I -    \mathbb{D}u ) e_N = (gh   - \sigma \Delta h)e_N & \text{on } \{y = \ell\} \\ 
u = 0 & \text{on } \{y= 0\}.
\end{cases}
\end{equation}
The corresponding linearized energy and dissipation are
\begin{equation}
E_{lin}(t)=\int_{\Omega} \half |u(z,t)|^2 dz +\half \int_\Sigma \left( g \abs{h(x,t)}^2 + \sigma \abs{\nab h(x,t)}^2 \right) dx 
\end{equation}
and
\begin{equation}
D_{lin}(t) =  \half\int_{\Omega} |{\mathbb D}u(z,t)|^2 dz,
\end{equation}
and they obey a linearized version of \eqref{ed_phsx}:
\begin{equation}\label{formal_ed}
 \frac{d}{dt} E_{lin}(t) + D_{lin}(t) =0.
\end{equation}
We can hope to gain some insight into the decay properties of solutions by using this to study the decay of $E_{lin}$.  

As defined, $D_{lin}$ provides no control of the linearized free surface function $h$, but we can use the equations \eqref{ns_linear} to try to gain some control.  Formally counting derivatives (in particular, ignoring the pressure for now), we expect to have the regularity relation 
\begin{equation}\label{formal_reg}
 gh - \sigma \Delta h \simeq \text{Tr}(\mathbb{D}u),
\end{equation}
where $\text{Tr}$ denotes the trace onto $\Gamma = \{y = \ell\}$ (see Theorem \ref{trace}).  

Now, when $\sigma >0$ the left side of  \eqref{formal_reg} is a second-order elliptic operator, so we expect that 
\begin{equation}
 \norm{h}_{H^{s+2}(\Sigma)} \ls \norm{\text{Tr}(\mathbb{D}u)}_{H^s(\Gamma)} \ls \norm{u}_{H^{s+3/2}(\Omega)}.
\end{equation}
Formally choosing $s + 3/2 = 1$ then shows that we expect 
\begin{equation}
 \norm{h}_{H^{3/2}(\Sigma)}^2 \ls D_{lin}
\end{equation}
and since 
\begin{equation}
 \half \int_\Sigma \left( g \abs{h(x)}^2 + \sigma \abs{\nab h(x)}^2 \right) dx \asymp \norm{h}_{H^1(\Sigma)}^2
\end{equation}
we conclude that it is plausible that the linearized dissipation is coercive over the linearized energy, i.e. 
\begin{equation}
 E_{lin} \ls D_{lin},
\end{equation}
which together with \eqref{formal_ed} would yield exponential decay of solutions with surface tension.  When $\Sigma = \mathbb{T}^2$, this argument can be made rigorous in a higher-regularity context (this is essentially the strategy of \cite{NishidaTTYYH2004}), but when $\Sigma = \R^2$ there are certain technical complications that obstruct exponential decay (see \cite{BealeNishida1985}).

On the other hand, when $\sigma =0$, the left side of \eqref{formal_reg} is a zeroth-order elliptic operator, so no regularity is gained.  A formal derivative count then shows that we expect 
\begin{equation}
 \norm{h}_{H^{s}(\Sigma)} \ls \norm{\text{Tr}(\mathbb{D}u)}_{H^s(\Gamma)} \ls \norm{u}_{H^{s+3/2}(\Omega)},
\end{equation}
which yields, upon formally setting $s =-1/2$, 
\begin{equation}
  \norm{h}_{H^{-1/2}(\Sigma)}^2 \ls D_{lin}.
\end{equation}
In the case without surface tension 
\begin{equation}
 \half \int_\Sigma   g \abs{h(x)}^2     dx \asymp \norm{h}_{L^2(\Sigma)}^2,
\end{equation}
and so we conclude that we cannot expect the dissipation functional to be coercive over the energy, and as such we cannot expect exponential decay.

%%%%%%%%%%%%%%%%%%%%%%%%%%%%%%%%%%%%%%%%%%%%%%%%%
%%%%%%%%%%%%%%%%%%%%%%%%%%%%%%%%%%%%%%%%%%%%%%%%%
\subsection{More general boundary operators }
%%%%%%%%%%%%%%%%%%%%%%%%%%%%%%%%%%%%%%%%%%%%%%%%% 
%%%%%%%%%%%%%%%%%%%%%%%%%%%%%%%%%%%%%%%%%%%%%%%%%

The above discussion highlights the role played by the order of the differential operator $g -\sigma \Delta$ in determining the decay properties of solutions to the linearized problem \eqref{ns_linear}.  Here $\sigma>0$ gives a second-order operator, while $\sigma =0$ gives a zeroth-order operator.  This suggests a natural question: what happens if we replace the operator with a more general one of fractional order?  The goal of this paper is to answer this question for the linearized problem.

To this end, we first introduce the more general boundary operator.  The operator $g - \sigma \Delta$ acts on the Fourier side (the Fourier transform is taken relative to the horizontal $x$ variable only, i.e. $\mathcal{F}: L^2(\Sigma) \to L^2(\hat{\Sigma})$ for $\hat{\Sigma}$ the dual group of $\Sigma$:  $\hat{\Sigma} = \R^{N-1}$ when $\Sigma = \R^{N-1}$ and $\hat{\Sigma} = \Z^{N-1}$ when $\Sigma = \T^{N-1}$) via 
\begin{equation}
\mathcal{F}(gh -\sigma \Delta h)(\xi) =    (g + 4\pi^2\sigma \abs{\xi}^2) \mathcal{F}h(\xi)  \text{ for } \xi \in \hat{\Sigma}.
\end{equation}
The essential features here are that the Fourier multiplier $g + 4\pi^2\sigma \abs{\xi}^2$ is real, bounded below by a positive constant, and grows quadratically for large $\xi$.   This suggests that we replace $g - \sigma \Delta$ by the more general Fourier multiplication operator, $M$, defined via
\begin{equation}\label{eq:Mmu}
\mathcal{F} M(h)(\xi)=\mu(\xi)\mathcal F h(\xi) \text{ for } \xi \in \hat{\Sigma}
\end{equation}
for a symbol function $\mu: \hat{\Sigma} \to \R$ satisfying 
\begin{equation}\label{mu_assump}
\theta \le \mu(\xi) \le \frac{1}{\theta}(1+|\xi|)^3 \text{ for all } \xi \in \hat{\Sigma}
\end{equation}
for some $\theta>0$.  Occasionally in the paper we will also demand  that 
\begin{equation}
\lim_{|\xi|\to\infty}\frac{\mu(\xi)}{|\xi|^3}=0,
\end{equation}
but we will always make this extra assumption clear. 

Since multiplication by $\xi$ on the frequency side corresponds to differentiation on the space side, we can think of $M$ as a generalized differential operator of order no more than $3$.  In particular, this choice of $\mu$ allows us to continuously interpolate between the second-order surface tension operator and the zeroth-order operator without surface tension.  It also allows us to exceed the order of the surface tension operator and go up to third order.  In principle we could allow $\mu$ to grow faster than cubically, but various calculations and statements of decay rates would become significantly more complicated, so we have focused our attention on the case of at most cubic growth.  We are most interested in the case 
\begin{equation}
\mu(\xi)=g+\sigma(2\pi|\xi|)^{2r} \text{ for } 0\le r\le 1,  
\end{equation}
which provides a very simple way of interpolating between no surface tension ($r=0$) and the presence of surface tension ($r=1$).  In this case we can write 
\begin{equation}
M(h)=gh+\sigma(-\Delta)^r h,
\end{equation}
as $(2\pi \abs{\xi})^{2r}$ is the standard symbol for the $r^{th}$ power of the negative Laplacian.

Two features of the operator $M$ defined by \eqref{eq:Mmu} are worth noting.  The first is that $\mu$ is not assumed to be a radial function, so the corresponding operator $M$ can be anisotropic.  The second is that in general the operator $M$ is nonlocal.  

Switching to the generalized boundary operator in \eqref{ns_linear} leads to the following system of equations:
\begin{equation}\label{eq:fullstrong}
\begin{cases}
\partial_t u +  \nabla p =   \Delta u  & \text{in }
\Omega \\ 
\diverge{u}=0 & \text{in }\Omega \\ 
\partial_t h = u_N   & 
\text{on } \{y = \ell\} \\ 
(p I -    \mathbb{D}u ) e_N = M(h)e_N & \text{on } \{y = \ell\} \\ 
u = 0 & \text{on } \{y= 0\}.
\end{cases}
\end{equation}
Our goal is to characterize the decay properties of this problem in terms of $M$, or equivalently $\mu$.

In the case that $\Sigma=\T^{N-1}$, there is an additional observation and convention we need.
Note that the condition $\diverge{u}=0$ guarantees by the divergence theorem that for all $y\in (0,\ell)$, 
\begin{equation}
\int_{\T^{N-1}} u_N(x,y,t) dx=\int_{\T^{N-1}\times(0,y)} \diverge u(z,t) dz=0
\end{equation}
and that
\begin{equation}\label{zero_avg_1}
\frac{d}{dt}\int_{\T^{N-1}} h(x,t) dx=\int_{\T^{N-1}} u_N(x,\ell,t) dx=0.
\end{equation}
We will assume that the initial data for the free surface, $h_0$, satisfies
\begin{equation}\label{zero_avg_2}
\int_{\T^{N-1}} h_0(x) dx=0, 
\end{equation}
since otherwise we should choose a different value of $\ell$. Thus we add the requirement that our solutions to \eqref{eq:fullstrong} also satisfy
\begin{equation}
\int_{\T^{N-1}} h(x,t) dx=0 \text{ for all } t \ge 0.
\end{equation}

There is an energy-dissipation structure associated with \eqref{eq:fullstrong}, namely if we set
\begin{equation}
\mathcal{E}(t) =\half \int_\Omega|u(z,t)|^2dz +\half \int_\Sigma M(h)(x,t)h(x,t) dx 
\end{equation}
and
\begin{equation}
\mathcal{D}(t) =\half \int_\Omega |{\mathbb D}u(z,t)|^2 dz, 
\end{equation}
then we have
\begin{equation}
\frac{d}{dt}\mathcal{E}(t)+\mathcal{D}(t)=0. 
\end{equation}
On the Fourier side, we may write 
\begin{equation}
\int_\Sigma M(h)h=\int_{\hat{\Sigma}} \mu(\xi) \abs{ \hat h(\xi)}^2 d\xi 
\end{equation}
so the energy is positive definite (when $\Sigma = \T^{N-1}$ the integral over $\hat{\Sigma}$ is really a sum).   Note that changing $M$ directly changes the structure of the energy, while leaving the structure of the dissipation unchanged.  However, when we seek to use the linearized equations to control $h$ in terms of the dissipation functional, we appeal to the equation 
\begin{equation}
M(h) =  p - \mathbb{D}u e_N \cdot e_N,
\end{equation}
which shows that these estimates of $h$ will also depend on the form of $M$.  Thus, we expect that the decay of solutions will depend on the balance of these features.

%%%%%%%%%%%%%%%%%%%%%%%%%%%%%%%%%%%%%%%%%%%%%%%%%
%%%%%%%%%%%%%%%%%%%%%%%%%%%%%%%%%%%%%%%%%%%%%%%%%
\subsection{Main results and discussion}
%%%%%%%%%%%%%%%%%%%%%%%%%%%%%%%%%%%%%%%%%%%%%%%%% 
%%%%%%%%%%%%%%%%%%%%%%%%%%%%%%%%%%%%%%%%%%%%%%%%%

In order to study the decay of the energy $\mathcal{E}$, we will decompose it into contributions from each individual Fourier mode $\xi \in \hat{\Sigma}$.  To this end we define the full energy to be
\begin{equation}\label{term:fullenergy} 
\mathcal{E}  =\half \int_\Omega|u|^2+\half \int_\Sigma M(h)h,
\end{equation}
and for each $\xi \in \hat{\Sigma}$ we define 
\begin{equation}\label{term:xienergy}
 \mathcal E_\xi  =\half \int_0^\ell|\hat u(\xi,y)|^2dy+\frac{\mu(\xi)}{2}|\hat h(\xi)|^2.
\end{equation}
These are clearly related via Fourier synthesis:
\begin{equation}
 \mathcal{E} = \int_{\hat{\Sigma}} \mathcal{E}_\xi d\xi.
\end{equation}

The main result of the paper is proved in Theorem \ref{thm:xidecayrate}, where we establish the decay properties of $\mathcal{E}_\xi$ for $\xi \neq 0$.  We prove that 
\begin{equation}
\mathcal E_\xi(t)\le b\mathcal E_\xi(0)\exp\left(-c\frac{|\xi|^2\mu(\xi)}{(1+|\xi|)^3}t\right)
\end{equation}
for constants $b,c >0$.  When $\Sigma = \T^{N-1}$ the zero mode is important, but it decays in a different manner: in Theorem \ref{thm:xidecayrate_zero} we show that $\mathcal{E}_0$ decays exponentially in this case.  These results highlight an essential feature of the linearized problem \eqref{eq:fullstrong}: there is a significant difference in the decay properties of the high frequency modes and the low frequency modes. Note that in the periodic case, $\Sigma = \T^{N-1}$, there is only one low frequency mode, namely the zero mode.  These decay estimates are essentially sharp in the sense that there exist solutions that achieve these decay rates.  We prove this in Theorem \ref{thm:ansatzfullbigxi} for high frequencies and in Theorem \ref{thm:ansatzfullsmallxi} for low frequencies, under some mild extra assumptions about $\mu$.
   
In the rest of our results we primarily specialize to the case $M = g + \sigma(-\Delta)^r$.  This gives a striking picture of the high-low split.  In this case our frequency-based decay results are summarized in the following table.
\begin{displaymath}
\begin{array}{| c | c | c | }
\hline
& 
\Sigma = \R^{N-1} & 
\Sigma = \T^{N-1}   \\ \hline
%%%%%%%%%%%%%%%%%
\textnormal{High frequency modes} &
\exp\left(-c|\xi|^{2r-1}t\right)  &
\exp\left(-c|\xi|^{2r-1}t\right)    \\ \hline
%%%%%%%%%%%%%%%%%
\textnormal{Low frequency modes} &
\exp\left(-c|\xi|^2t\right)  &
 \exp(-\delta t)     \\ \hline
%%%%%%%%%%%%%%%%%
\end{array}
\end{displaymath}

As mentioned above, in the periodic case $\Sigma = \T^{N-1}$ there is only one low frequency mode, and it decays exponentially.  This means that the decay of the nonzero modes completely determines the decay of the full energy $\mathcal{E}$.  When $1/2 \le r \le 1$ these frequencies decay exponentially, and a Fourier synthesis then allows us to prove in Theorem \ref{thm:torusdecayexponential} that $\mathcal{E}$ also decays exponentially.  When $0 \le r < 1/2$ the nonzero frequencies do not decay exponentially. In Theorem \ref{thm:torusdecayalgebraic} we show that we can still prove that solutions decay, but that they do so at an algebraic rate tied to the regularity of the initial data and to the difference $1/2-r$.  In particular, 
\begin{equation}
\mathcal{E}(t) \ls K(1+t)^{-s/(1/2 -r)}
\end{equation}
where $s >0$ is a Sobolev regularity index associated to the data and $K$ depends on the data (up to regularity $s$).  This means that the more regular the data are, the faster the solution decays.  This is almost exponential decay in the same sense as proved in \cite{GuoTice2013} when $r=0$, i.e. without surface tension.  Moreover, the closer $r$ is to $1/2$, the more decay is produced by gains in regularity of the data.

This analysis shows that there is a sharp transition that occurs at the critical index $r =1/2$, with solutions decaying almost exponentially for $r < 1/2$ and exponentially for $1/2 \le r$.  This suggests a more detailed study of the transition is in order.  In Theorem \ref{thm:transitiondecay} we consider the cases in which the symbol $\mu$ is of the form 
\begin{equation}
\mu(\xi)=g + 2\pi \sigma \frac{|\xi|}{(\log|\xi|)^\alpha} \text{ or } 
 \mu(\xi)=g + 2\pi \sigma \frac{|\xi|}{(\log\log|\xi|)^\alpha} 
\end{equation}
for large $\xi$.  These provide a simple way of zooming in around the index $r=1/2$ to further study the transition.  In the former case we prove that 
\begin{equation}\label{eq:main_refine_1}
\mathcal E(t)\le K  \exp\left(- c t^{\frac{1}{1+\alpha}}\right),
\end{equation}
and in the latter case we prove that 
\begin{equation}\label{eq:main_refine_2}
\mathcal E(t)\le K  \exp\left(- c \frac{t}{(\log t)^{\alpha}}\right)
\end{equation}
where $K$ is a constant depending on the data and $c>0$ is a constant depending on the parameters.  This provides a more refined picture of what happens near the critical index: solutions shift from almost exponential decay to weaker forms of exponential decay, as in \eqref{eq:main_refine_1} and \eqref{eq:main_refine_2}.

The disparity between the high and low frequency decay rates plays a more serious role in determining the decay of the full energy $\mathcal{E}$ in the non-periodic case $\Sigma = \R^{N-1}$.  This is due to the obvious fact that there are nonzero low frequency modes.  The decay properties of the high frequency modes remain the same as in the periodic case, but the slower decay of the low frequencies slows the overall decay rate of $\mathcal{E}$, which we prove in Theorems \ref{thm:nonper_L2_decay} and \ref{thm:nonper_L1_decay}.  In these results we use assumptions on the initial data to guarantee quantitative decay rates for the low frequencies.  In Theorem \ref{thm:nonper_L2_decay} we use $L^2-$based spaces and the Riesz transform, as done in the analysis of the problem without surface tension ($r=0$) in \cite{GuoTice2013a2}.  In contrast, in Theorem \ref{thm:nonper_L1_decay} we use $L^1-$based assumptions as done for the problem with surface tension in \cite{BealeNishida1985}.  These both yield fixed algebraic decay rates.

The following table summarizes our decay rates in terms of $\Sigma$ and the operator $M = g + \sigma(-\Delta)^r$. 
\begin{displaymath}
\begin{array}{| c | c | c | }
\hline
& 
\Sigma = \R^{N-1} & 
\Sigma = \T^{N-1}   \\ \hline
%%%%%%%%%%%%%%%%%
 1/2 \le r \le 1  &
\textnormal{Algebraic decay}  &
\textnormal{Exponential decay}    \\ \hline
%%%%%%%%%%%%%%%%%
0 \le r < 1/2 &
\textnormal{Algebraic decay}  &
\textnormal{Almost exponential decay}   \\ \hline
%%%%%%%%%%%%%%%%%
\end{array}
\end{displaymath}

%%%%%%%%%%%%%%%%%%%%%%%%%%%%%%%%%%%%%%%%%%%%%%%%%
%%%%%%%%%%%%%%%%%%%%%%%%%%%%%%%%%%%%%%%%%%%%%%%%%
%%%%%%%%%%%%%%%%%%%%%%%%%%%%%%%%%%%%%%%%%%%%%%%%%
\section{Preliminaries }
%%%%%%%%%%%%%%%%%%%%%%%%%%%%%%%%%%%%%%%%%%%%%%%%% 
%%%%%%%%%%%%%%%%%%%%%%%%%%%%%%%%%%%%%%%%%%%%%%%%%
%%%%%%%%%%%%%%%%%%%%%%%%%%%%%%%%%%%%%%%%%%%%%%%%%

In this section we collect a few tools and notational conventions used throughout the remainder of the paper.

%%%%%%%%%%%%%%%%%%%%%%%%%%%%%%%%%%%%%%%%%%%%%%%%%
%%%%%%%%%%%%%%%%%%%%%%%%%%%%%%%%%%%%%%%%%%%%%%%%%
\subsection{Fourier transform}
%%%%%%%%%%%%%%%%%%%%%%%%%%%%%%%%%%%%%%%%%%%%%%%%% 
%%%%%%%%%%%%%%%%%%%%%%%%%%%%%%%%%%%%%%%%%%%%%%%%%

Whenever we apply the Fourier transform, it will be with respect to the horizontal variable $x\in \R^{N-1}$.  For any function that is at least horizontally $L^2$, the Fourier transform will be defined in the standard way, using the convention for the Fourier transform that
\begin{equation}
\mathcal F f(\xi)=\int_{\Sigma} f(x)e^{-2\pi i\xi\cdot x}dx
\end{equation}
for $\xi \in \hat{\Sigma}$.  Here the dual group is 
\begin{equation}
 \hat{\Sigma} = 
\begin{cases}
\R^{N-1} &\text{if } \Sigma = \R^{N-1} \\
\Z^{N-1} &\text{if } \Sigma = \T^{N-1}.
\end{cases}
\end{equation}
We will use the fact that the Fourier transform with respect to $x$ commutes with derivatives with respect to $y$ (both classical and weak derivatives). 

We will also use the Sobolev spaces $H^s(\Sigma)$, defined through the norm
\begin{equation}
||f||_{H^s(\Sigma)}^2=\int_{{\hat\Sigma}}(1+|\xi|^2)^{s}|\mathcal F f(\xi)|^2 d\xi 
\end{equation}
and the space $H^s_y(L^2_x)$, defined by
\begin{equation}
||f||_{H^s_x(L^2_y)}^2=\int_0^\ell \int_{{\hat\Sigma}}(1+|\xi|^2)^{s}|\mathcal F f(\xi,y)|^2 d\xi dy.
\end{equation}

%%%%%%%%%%%%%%%%%%%%%%%%%%%%%%%%%%%%%%%%%%%%%%%%%
%%%%%%%%%%%%%%%%%%%%%%%%%%%%%%%%%%%%%%%%%%%%%%%%%
\subsection{Function spaces and weak forms}
%%%%%%%%%%%%%%%%%%%%%%%%%%%%%%%%%%%%%%%%%%%%%%%%% 
%%%%%%%%%%%%%%%%%%%%%%%%%%%%%%%%%%%%%%%%%%%%%%%%%

Here we quickly survey some of the function spaces used in defining a weak formulation of \eqref{eq:fullstrong}.  We will often use $L^2_T(X)$ to mean $L^2([0,T];X)$ and $L^2_y(X)$ to mean $L^2([0,\ell];X)$.  We define $H^1((0,\ell))$ to be the Sobolev space of functions $u\in L^2((0,\ell))$ with weak derivatives $u' = \p_y u \in L^2((0,\ell))$.  We define $H^1(\Omega)$ to be the Sobolev space of functions $u\in L^2(\Omega)$ with weak derivatives $\nabla u\in L^2(\Omega)$.  We equip both of these with the standard norms and inner-products.  We will use $(\cdot,\cdot)$ to denote an $L^2$ inner product and $\langle\cdot,\cdot\rangle$ to denote an $H^1$ inner product.

In order to accommodate the no-slip condition at the bottom of $\Omega$, it is convenient to define the space
\begin{equation}
\Hbot((0,\ell))=\{u\in H^1((0,\ell))\mid u(0)=0\} 
\end{equation}
and endow it with the inner product $\langle u,v\rangle_{\Hbot}=(u',v')$.  Similarly, we define
\begin{equation}
\Hbot(\Omega)=\{u\in H^1(\Omega)\mid u=0\text{ on }\{y=0\}\},
\end{equation}
endowed with the inner product $\langle u,v\rangle_{\Hbot}=(\nabla u,\nabla v)$. By the Poincar\'e inequality (see Theorem~\ref{poincare} in the appendix), these give norms equivalent to the standard one on $H^1$. Given $\xi\in\R^{N-1}$, let
\begin{equation}
\Hxi=\{u\in H^1((0,\ell))\mid u=0\text{ on }\{y=0\}\}
\end{equation}
with the inner product
\begin{equation}
\langle u,v\rangle_{\Hxi}=(u', v')+4\pi^2|\xi|^2(u,v). 
\end{equation}
This is essentially the $\Hbot(\Omega)$ inner product for a particular Fourier frequency, since it is easy to see that
\begin{equation}
\langle u,v\rangle_{\Hbot(\Omega)}=\int \langle\hat u(\xi),\hat v(\xi)\rangle_{\Hxi}d\xi. 
\end{equation}
We will use $[\cdot,\cdot]$ to denote the pairing of $H^1_\xi$ with its dual and the pairing of $\Hbot$ with its dual. The element of the dual will be the first argument.

On $\Omega=\Sigma\times (0,\ell)$, we also define
\begin{equation}
\Hhel=\{u\in \Hbot(\Omega;\R^N)\mid u=0\text{ on }\{y=0\}, \text{div }u=0\}
\end{equation}
with the inner product $\langle u,v\rangle_{\Hhel}=(\nabla u,\nabla v)$. Since this is inherited from its structure as $\Hbot(\Omega;\R^N)$, this inner product gives a norm equivalent to the $H^1$ norm. Also similarly to how the $\Hxi$ is essentially the $\Hbot(\Omega)$ inner product for a particular frequency, we can define a space $\Hxihel$ to be essentially the $\Hhel(\Omega)$ inner product for a particular frequency, namely
\begin{equation}
\Hxihel=\left\{u:(0,\ell)\to \C^N\,\middle|\,u\in (\Hxi)^N, u_N'+2\pi i\sum_{i=1}^{N-1}\xi_i u_i=0\right\}
\end{equation}
with the inner product inherited from $(\Hxi)^N$.

We will need to consider functions that are in $\Hhel$ and whose time derivatives are in $\Hhel^*$. For this to make sense, we need an embedding $\Hhel\hookrightarrow\Hhel^*$. To get such an embedding, we note that $\Hhel$ is dense in its $L^2$ closure $\overline{\Hhel}$ (with the $L^2$ norm). Then $\overline{\Hhel}$ is its own dual under the $L^2$ inner product, so we get an embedding $\overline{\Hhel}\hookrightarrow\Hhel^*$. Composing, we get an embedding $\Hhel\hookrightarrow\Hhel^*$. The same construction of embeddings will be used for $\Hbot, \Hxi$, and $\Hxihel$.

We have the following weak formulation of \eqref{eq:fullstrong}.  We say that $(u,h)$ is a weak solution to \eqref{eq:fullstrong} if $u\in L^2_{\text{loc}}([0,\infty);\Hhel)$, $\dt u \in L^2_{\text{loc}}([0,\infty);(\Hhel)^\ast)$, $h \in L^2_{\text{loc}}([0,\infty);L^2(\Sigma))$, $M(h) \in L^2_{\text{loc}}([0,\infty);L^2(\Sigma))$, $\dt h \in L^2_{\text{loc}}([0,\infty);H^{1/2}(\Sigma))$, and 
\begin{equation}\label{eq:fullweak}
\left\{\text{
\begin{tabular}{ l l }
  $[\partial_tu,v]_{\Hhel^*,\Hhel}+\half({\mathbb D}u,{\mathbb D}v)+\int_{\{y=\ell\}}M(h)v_N=0$ &$\forall\,v\in L^2_T(\Hhel)$, for a.\,e. $t$ \\
	$\Div\ u=0$ \\
	$\partial_t h=u_N$ & on $\{y=\ell\}$\\
	$u=u_0$ & at $t=0$\\
	$h=h_0$ & at $t=0$\\
\end{tabular}}\right.
\end{equation}
with the extra condition 
\begin{equation}\label{eq_zero_avg}
\int_{\T^{N-1}} h(x,t) dx=0 
\end{equation}
when $\Sigma= \T^{N-1}$.

%%%%%%%%%%%%%%%%%%%%%%%%%%%%%%%%%%%%%%%%%%%%%%%%%
%%%%%%%%%%%%%%%%%%%%%%%%%%%%%%%%%%%%%%%%%%%%%%%%%
\subsection{Some useful estimates  }
%%%%%%%%%%%%%%%%%%%%%%%%%%%%%%%%%%%%%%%%%%%%%%%%% 
%%%%%%%%%%%%%%%%%%%%%%%%%%%%%%%%%%%%%%%%%%%%%%%%%

Here we derive some estimates that will later be useful in various calculations.

Let $v:(0,\ell)\to\R$ satisfy $v\in H^1_\xi$. Then
\begin{equation}
||v||_{\Hxi}^2=||v'||_{L^2}^2+4\pi^2|\xi|^2 ||v||_{L^2}^2\gesim (1+|\xi|^2)||v||_{L^2}^2
\end{equation}
by the Poincar\'e inequality, Theorem \ref{poincare} in the appendix. Thus
\begin{equation}\label{eq:HxiL2bound}
||v||_{L^2}\lesim\frac{||v||_{\Hxi}}{1+|\xi|}.
\end{equation}

For $\varphi\in C^\infty([0,\ell])$ with $\varphi(0)=0$, we may compute
\begin{equation}
\begin{split}
(1+|\xi|)\varphi(\ell)^2&\le\left|\int_0^\ell \varphi'(y) dy\right|^2+|\xi|\int_0^\ell 2\varphi'(y)\varphi(y) dy\\
&\le \ell\int_0^\ell \varphi'(y)^2 dy+2|\xi|\sqrt{\int_0^\ell \varphi'(y)^2 dy}\sqrt{\int_0^\ell \varphi(y)^2 dy}\\
&\le||\varphi||_{\Hxi}\left(\ell||\varphi||_{\Hxi}+\frac{1}{\pi}||\varphi||_{\Hxi} \right)\le(\ell+1)||\varphi||_{\Hxi}^2,
\end{split}
\end{equation}
so
\begin{equation}\label{eq:endpointestimate}
\varphi(\ell)^2\le\frac{1+\ell}{1+|\xi|} ||\varphi||_{\Hxi}^2
\end{equation}
for all $\xi$. We use density of such smooth functions in $\Hxi$ to extend this to all $\varphi\in\Hxi$. Using this, we get 
\begin{equation}\label{eq:randestimate1}
a\varphi(\ell)\ge -\frac{1+\ell}{1+|\xi|}|a|^2-\frac{1+|\xi|}{4(1+\ell)}\varphi(\ell)^2\ge -\frac{1+\ell}{1+|\xi|}|a|^2-\frac{1}{4}||\varphi||_{\Hxi}^2.
\end{equation}

For the remaining three estimates, we assume that $\Sigma=\T^{N-1}$, consider $w\in \Hbot(\Omega)$, and fix some $\xi\in \Z^{N-1}$. We then define $\tilde u:\Omega\to\R$ via $\tilde u(x,y)=\hat u(\xi,y,t)e^{2\pi i\xi\cdot x}$ (that is, $\tilde u$ is just a single Fourier mode of $u$). Then 
\begin{equation}
\widehat{\nabla u}(\xi)e^{2\pi i\xi\cdot x}=\nabla \tilde u(x,y) \text{ and }\widehat{\mathbb D u}(\xi)e^{2\pi i\xi\cdot x}=\mathbb D \tilde u(x,y). 
\end{equation}
If we define $\varphi(y)=\hat u(\xi,y)$, then \eqref{eq:endpointestimate} gives us
\begin{equation}\label{eq:randestimate2}
\begin{split}
|\hat u(\xi,\ell)|^2&\le\frac{1+\ell}{1+|\xi|} ||\hat u(\xi,\cdot)||_{\Hxi}^2=\frac{1+\ell}{1+|\xi|} ||\tilde u||_{\Hbot(\Omega)}^2=\frac{1+\ell}{1+|\xi|} ||\nabla \tilde u||_{L^2(\Omega)}^2\nonumber\\
|\hat u(\xi,\ell)|^2&\le\frac{1+\ell}{1+|\xi|} ||\widehat{\nabla u}(\xi)||_{L^2(0,\ell)}^2.
\end{split}
\end{equation}

We can also use the Poincar\'e inequality on $(0,\ell)$ to see that
\begin{equation}
||\widehat{\nabla u}(\xi,\cdot)||_{L^2(0,\ell)}^2=||\hat u(\xi,\cdot)||_{\Hxi}^2\gesim \left|\left|\frac{d}{dy}\hat u(\xi,\cdot)\right|\right|_{L^2(0,\ell)}^2+|\xi|^2||\hat u(\xi,\cdot)||_{L^2(0,\ell)}^2\gesim (1+|\xi|^2)||\hat u(\xi,\cdot)||_{L^2(0,\ell)}^2,
\end{equation}
so
\begin{equation}\label{ineq:fourierpoincare}
||\hat u(\xi,\cdot)||_{L^2(0,\ell)}\lesim\frac{1}{1+|\xi|}||\widehat{\nabla u}(\xi,\cdot)||.
\end{equation}
We may also use Korn's inequality,  Theorem \ref{korn}, to calculate
\begin{equation}\label{ineq:fourierkorn}
\begin{split}
||\widehat{\nabla u}(\xi)||_{L^2(0,\ell)}^2&=||\widehat{\nabla u}(\xi)e^{2\pi i \xi\cdot x}||_{L^2(\Omega)}^2=||\nabla \tilde u||_{L^2(\Omega)}^2\lesim ||\mathbb D \tilde u||_{L^2(\Omega)}^2=||\widehat{\mathbb D u}(\xi)e^{2\pi i \xi\cdot x}||_{L^2(\Omega)}^2\\
||\widehat{\nabla u}(\xi)||_{L^2(0,\ell)}^2&\lesim||\widehat{\mathbb Du}(\xi)||_{L^2(0,\ell)}^2,
\end{split}
\end{equation}
where the constant suppressed by $\lesim$ may depend on $\ell$, but not on $\xi$.

%%%%%%%%%%%%%%%%%%%%%%%%%%%%%%%%%%%%%%%%%%%%%%%%%
%%%%%%%%%%%%%%%%%%%%%%%%%%%%%%%%%%%%%%%%%%%%%%%%%
%%%%%%%%%%%%%%%%%%%%%%%%%%%%%%%%%%%%%%%%%%%%%%%%%
\section{Decay analysis at a fixed frequency }
%%%%%%%%%%%%%%%%%%%%%%%%%%%%%%%%%%%%%%%%%%%%%%%%% 
%%%%%%%%%%%%%%%%%%%%%%%%%%%%%%%%%%%%%%%%%%%%%%%%%
%%%%%%%%%%%%%%%%%%%%%%%%%%%%%%%%%%%%%%%%%%%%%%%%%

Our goal in this section is to study the decay properties of the energy $\mathcal{E}_\xi$, defined by \eqref{term:xienergy}, as a function of $\xi$.  We will first derive decay estimates and then show that these are sharp in some cases.

%%%%%%%%%%%%%%%%%%%%%%%%%%%%%%%%%%%%%%%%%%%%%%%%%
%%%%%%%%%%%%%%%%%%%%%%%%%%%%%%%%%%%%%%%%%%%%%%%%%
\subsection{Decay of Fourier modes }
%%%%%%%%%%%%%%%%%%%%%%%%%%%%%%%%%%%%%%%%%%%%%%%%% 
%%%%%%%%%%%%%%%%%%%%%%%%%%%%%%%%%%%%%%%%%%%%%%%%%

We now turn to the question of determining the rate of decay of $\mathcal{E}_\xi(t)$.  In order to treat the cases $\Sigma = \T^{N-1}$ and $\Sigma = \R^{N-1}$ in a unified manner, we begin by examining $\xi \neq 0$.

\begin{thm}\label{thm:xidecayrate}
Suppose that $u$ is a weak solution to $\eqref{eq:fullweak}$.  There exist constants $b,c>0$ such that for almost every $\xi \in \hat{\Sigma} \backslash \{0\}$ we have that
\begin{equation}\label{eq:decayrate}
\mathcal E_\xi(t)\le b\mathcal E_\xi(0)\exp\left(-c\frac{|\xi|^2\mu(\xi)}{(1+|\xi|)^3}t\right),
\end{equation}
where $\mathcal{E}_\xi$ is defined by \eqref{term:xienergy}.
\end{thm}
\begin{proof}
Throughout the proof we will abbreviate $L^2_y = L^2((0,\ell))$.  We will initially prove the result in the case $\Sigma = \T^{N-1}$, in which case $\hat{\Sigma} = \Z^{N-1}$.  Fix $\xi \in \hat{\Sigma} \backslash \{0\}$.  Upon complexifying the weak formulation of \eqref{eq:fullweak}, we have that
\begin{equation}\label{eq:weakrestate}
[\partial_t u,v]_{(\Hhel)^*,\Hhel}+\half({\mathbb D}u,{\mathbb D}v)_{L^2(\Omega)}+\int_{\{y=\ell\}} M(h)\overline{v_N} dx=0\qquad\forall\,v\in\ L^2_T(\Hhel) \text{ for a.\,e. }t,
\end{equation}
where here the test functions $v$ can take values in $\C^N$.  Here we introduced the complexification for convenience in working with Fourier modes.  Note that in \eqref{eq:weakrestate} we are taking $L^2$ inner products and dual pairings in spaces of complex-valued functions, so we have to be careful about complex conjugates.

Consider $V : \Omega \times [0,\infty) \to \C^N$ defined by 
\begin{equation}
 V(x,y,t) = \hat u(\xi,y,t)e^{2\pi i\xi\cdot x},
\end{equation}
which is just the component of $u$ from a single Fourier mode at spatial frequency $\xi$.  From the commutativity of the Fourier transform with weak derivatives and from the fact that $\hat V(\zeta,y,t)=0$ for $\zeta \ne \xi$, we get 
\begin{equation}
\begin{split}
[\partial_t u,V]_{(\Hhel)^*,\Hhel}& +\overline{[\partial_t u,V]_{(\Hhel)^*,\Hhel}} \\
%&=\sum_{\zeta\in\Z^{N-1}}[\widehat{\partial_t  u}(\zeta,\cdot,t),\hat V(\zeta,\cdot,t)]_{(\Hxihel)^*,%\Hxihel}+\overline{[\widehat{\partial_t  u}(\zeta,\cdot,t),\hat V(\zeta,\cdot,t)]_{(\Hxihel)^*,\Hxihel}}\\
&=[\partial_t  \hat u(\xi,\cdot,t),\hat V(\xi,\cdot,t)]_{(\Hxihel)^*,\Hxihel}+\overline{[\partial_t  \hat u(\xi,\cdot,t),\hat V(\xi,\cdot,t)]_{(\Hxihel)^*,\Hxihel}}.
\end{split}
\end{equation}
Applying Theorem~\ref{thm:CL2} from the appendix shows that
\begin{equation}
[\partial_t u,V]_{(\Hhel)^*,\Hhel}+\overline{[\partial_t u,V]_{(\Hhel)^*,\Hhel}}=\frac{d}{dt}||\hat u(\xi,\cdot,t)||_{L^2((0,\ell))}^2.
\end{equation}
Also, 
\begin{equation}
{\mathbb D}V(x,y,t)=\widehat{{\mathbb D}u}(\xi,y,t)e^{2\pi i\xi\cdot x}, 
\end{equation}
so we may compute
\begin{equation}
({\mathbb D}u,{\mathbb D}V)=||\widehat{{\mathbb D}u}(\xi,\cdot,t)||_{L^2((0,\ell))}^2.
\end{equation}
Finally,
\begin{equation}
\int_{\{y=\ell\}} M(h)\overline{V_N} dx=\sum_{\zeta\in\Z^{N-1}}\mu(\zeta) \hat h(\zeta,t)\overline{\hat V(\zeta,\ell,t)}=\mu(\xi) \hat h(\xi,t)\overline{\hat u(\xi,\ell,t)}=\mu(\xi) \hat h(\xi,t)\widehat{\partial_t\bar h}(\xi,t).
\end{equation}
Adding the equation above to its complex conjugate then gives
\begin{equation}
\int_{\{y=\ell\}} M(h)\overline{V_N} dx+\overline{\int_{\{y=\ell\}} M(h)\overline{V_N} dx}=\mu(\xi)\hat h(\xi,t)\widehat{\partial_t\bar h}(\xi,t)+\mu(\xi)\overline{\hat h}(\xi,t)\widehat{\partial_th}(\xi)=\frac{d}{dt}\mu(\xi)|\hat h(\xi,t)|^2,
\end{equation}
where in the last equality we have again used the commutativity of weak derivatives with the Fourier transform. Thus adding \eqref{eq:weakrestate} to its complex conjugate and dividing by 2 gives
\begin{equation}\label{eq:firsteq}
0=\frac{d}{dt}\left(\half||\hat u(\xi,\cdot,t)||_{L^2_y}^2+\frac{\mu(\xi)}{2}|\hat h(\xi,\cdot,t)|^2\right)+\half||\widehat{{\mathbb D}u}(\xi,\cdot,t)||_{L^2_y}^2
\end{equation}
for $t >0$.

We will now construct another function $v: \Omega \times [0,\infty) \to \C^N$ to plug into the weak formulation.  First, we set
\begin{equation}
a_\xi(y)=\frac{\sinh(|\xi|y)}{\sinh(|\xi|\ell)}. 
\end{equation}
Note that $a_\xi\in C^\infty([0,\ell])$ satisfies $a_\xi(\ell)=1$, $a_\xi(0)=0$, as well as the bounds
\begin{equation}
\int_0^\ell |a_\xi(y)|^2 dy\lesim \frac{1}{|\xi|+1}, \int_0^\ell |a'_\xi(y)|^2 dy\lesim 1+|\xi|, \int_0^\ell |a''_\xi(y)|^2 dy\lesim (1+|\xi|)^3. 
\end{equation}
We then let $w:(0,\ell)\to \C^N$ be given by 
\begin{equation}\label{eq:betawdef}
w(y)=\left(\frac{i\xi}{2\pi|\xi|^2} a'_\xi(y),a_\xi(y)\right),
\end{equation}
where $' = d/dy$, and then we let $v: \Omega \times [0,\infty) \to \C^N$ be given by 
\begin{equation}
 v(x,y,t) = \hat h(\xi,t)e^{2\pi i\xi\cdot x}w(y).
\end{equation}
Clearly $v(\cdot,t)\in \Hhel$. Then $\eqref{eq:weakrestate}$ gives us
\begin{equation}\label{eq:weakplugin}
0=[\partial_t u,v]_{(\Hbot)^*,\Hbot}+\half({\mathbb D}u,{\mathbb D}v)_{L^2(\Omega)}+\sum_{\Z^{N-1}} \mu(\zeta,t) \hat h(\zeta,t)\overline{\hat v_N}(\zeta,\ell,t) d\zeta.
\end{equation}
Also, from Theorem~\ref{thm:CL2} in the appendix and the polarization identity, we have
\begin{equation}
\frac{d}{dt}(u,v)_{L^2(\Omega)}=[\partial_t u,v]_{(\Hbot)^*,\Hbot}+\overline{[\partial_t v,u]_{(\Hbot)^*,\Hbot}}=[\partial_t u,v]_{(\Hbot)^*,\Hbot}+(u,\partial_t v)_{L^2(\Omega)}.
\end{equation}
Plugging this into $\eqref{eq:weakplugin}$ and using the formula for $\hat v_N$, we get
\begin{equation} 
0=\frac{d}{dt}(u,v)_{L^2(\Omega)}-(u,\partial_t v)_{L^2(\Omega)}+\half({\mathbb D}u,{\mathbb D}v)_{L^2(\Omega)}+\mu(\xi) |\hat h(\xi,t)|^2.
\end{equation}
Applying Parseval's theorem to this identity then shows that 
\begin{equation}\label{eq:secondeq}
0=\frac{d}{dt}(\hat{u}(\xi,\cdot,t),\hat{v}(\xi,\cdot,t))_{L^2_y}-(\hat{u}(\xi,\cdot,t),\partial_t \hat{v}(\xi,\cdot,t))_{L^2_y}+\half(\widehat{{\mathbb D}u}(\xi,\cdot,t),\widehat{{\mathbb D}v}(\xi,\cdot,t))_{L^2_y}+\mu(\xi) |\hat h(\xi,t)|^2.
\end{equation}

We will take a positive linear combination of $\eqref{eq:firsteq}$ and $\eqref{eq:secondeq}$, weighing the latter by a factor $\beta$, and we will use the fact that $\partial_t\hat h(\xi,t)=\hat u_N(\xi,\ell,t)$ (which holds since weak derivatives commute with the Fourier transform) to get
\begin{multline}\label{eq:fullenergydissipation}
0=\frac{d}{dt}\left(\half||\hat u(\xi,\cdot,t)||_{L^2_y}^2+\frac{\mu(\xi)}{2}|\hat h(\xi,t)|^2+\beta (\hat{u}(\xi,\cdot,t),\hat{v}(\xi,\cdot,t))_{L^2_y} \right)  \\
+\half||\widehat{{\mathbb D}u}(\xi,\cdot,t)||_{L^2_y}^2-\beta\hat u_N(\xi,\ell,t)(\hat u(\xi,\cdot,t),w)_{L^2_y}+\frac{\beta}{2}(\widehat{{\mathbb D}u}(\xi,\cdot,t),\widehat{{\mathbb D}v}(\xi,\cdot,t))_{L^2(\Omega)}+\beta\mu(\xi) |\hat h(\xi,t)|^2.
\end{multline}
We will employ this equality to prove the decay result, but to do so we must first absorb various terms.

We claim we can choose some $\beta=\beta(|\xi|)$ so that the following hold:
\begin{align}
|\beta (\hat{u}(\xi,\cdot,t),\hat{v}(\xi,\cdot,t))_{L^2_y} |&\le\frac{1}{4}||\hat u(\xi,\cdot,t) ||_{L^2_y}^2+\frac{\mu(\xi)}{4} |\hat h(\xi,t)|^2\label{eq:betarec1}\\
\left|\frac{\beta}{2}(\widehat{{\mathbb D}u}(\xi,\cdot,t),\widehat{{\mathbb D}v}(\xi,\cdot,t))_{L^2_y}\right|&\le \frac{1}{8}||\widehat{{\mathbb D}u}(\xi,\cdot,t)||_{L^2_y}^2+\frac{\beta\mu(\xi)}{2} |\hat h(\xi,t)|^2\label{eq:betarec2}\\
                    |\beta\hat u_N(\xi,\ell,t)(\hat u(\xi,\cdot,t),w)_{L^2_y}|&\le \frac{1}{8}||\widehat{{\mathbb D}u}(\xi,\cdot,t)||_{L^2_y}^2.\label{eq:betarec3}
\end{align}
We will now proceed to determine what upper bounds on $\beta$ these conditions require.  We first compute 
\begin{equation}
\begin{split}
|(\hat u(\xi,\cdot,t),\hat v(\xi,\cdot,t))_{L^2_y}| &\le\frac{\alpha}{2}||\hat u(\xi,\cdot,t)||_{L^2_y}^2+\frac{1}{2\alpha}||\hat v(\xi,\cdot,t)||_{L^2_y}^2 \\
& \le \frac{\alpha}{2}||\hat u(\xi,\cdot,t)||_{L^2_y}^2+\frac{|\hat h(\xi,t)|^2}{2\alpha}\int_0^\ell |w(y)|^2 dy\\
&\lesim \alpha||\hat u(\xi,\cdot,t)||_{L^2_y}^2+\frac{|\hat h(\xi,t)|^2}{\alpha}\left(\frac{1+|\xi|}{|\xi|^2}+\frac{1}{1+|\xi|}\right) \\
&\lesim \alpha||\hat u(\xi,\cdot,t)||_{L^2_y}^2+\frac{1+|\xi|}{\alpha|\xi|^2}|\hat h(\xi,t)|^2,
\end{split}
\end{equation}
so in order to satisfy \eqref{eq:betarec1}, we can take 
\begin{equation}
\alpha^2=\frac{1+|\xi|}{|\xi|^2\mu(\xi)} \text{ and }\beta\le c\sqrt{\frac{|\xi|^2\mu(\xi)}{1+|\xi|}}  
\end{equation}
for a sufficiently small constant $c>0$.

A similar computation shows that
\begin{equation}
\begin{split}
|(\widehat{{\mathbb D}u(\xi,\cdot,t)},\widehat{{\mathbb D}v}(\xi,\cdot,t))_{L^2_y}| 
&\le \frac{\alpha}{2}||\widehat{{\mathbb D}u}(\xi,\cdot,t)||_{L^2_y}^2+\frac{1}{2\alpha}||\widehat{{\mathbb D}v}(\xi,\cdot,t)||_{L^2_y}^2 \\
&\lesim \alpha||\widehat{{\mathbb D}u}(\xi,\cdot,t)||_{L^2_y}^2+\frac{1}{\alpha}||\widehat{\nabla v}(\xi,\cdot,t)||_{L^2_y}^2\\
&\lesim \alpha||\widehat{{\mathbb D}u}(\xi,\cdot,t)||_{L^2_y}^2+\frac{1}{\alpha}|\hat h(\xi,t)|^2\int_0^\ell |w'(y)|^2 +4\pi^2|\xi|^2 |w(y)|^2dy\\
&\lesim \alpha||\widehat{{\mathbb D}u}(\xi,\cdot,t)||_{L^2_y}^2+\frac{1}{\alpha}|\hat h(\xi,t)|^2\left(\frac{(1+|\xi|)^3}{|\xi|^2}+1+|\xi|+\frac{|\xi|^2}{1+|\xi|}\right)\\
&\lesim \alpha||\widehat{{\mathbb D}u}(\xi,t)||_{L^2_y}^2+\frac{(1+|\xi|)^3}{|\xi|^2\alpha}|\hat h(\xi,t)|^2,
\end{split}
\end{equation}
so in order to satisfy \eqref{eq:betarec2}, we can take 
\begin{equation}
\alpha=\frac{2(1+|\xi|)^3}{|\xi|^2\mu(\xi)\sqrt{c}} \text{ and } \beta\le c\frac{|\xi|^2\mu(\xi)}{(1+|\xi|)^3} 
\end{equation}
for a sufficiently small constant $c>0$.

Finally, using \eqref{eq:randestimate2} and \eqref{ineq:fourierpoincare} for the second inequality and \eqref{ineq:fourierkorn} on the fourth inequality, we get
\begin{equation}
\begin{split}
|\hat u_N(\xi,\ell,t)(\hat u(\xi,\cdot,t),w)_{L^2_y}|&\le|\hat u_N(\xi,\ell,t)|||\hat u(\xi,\cdot,t)||_{L^2_y}||w||_{L^2_y}\\
&\lesim\frac{1}{\sqrt{1+|\xi|}}||\widehat{\nabla u}(\xi,\cdot,t)||_{L^2_y}\left(\frac{1}{1+|\xi|}||\widehat{\nabla u}(\xi,\cdot,t)||_{L^2_y}\right)\left(\frac{1}{1+|\xi|}+\frac{1+|\xi|}{|\xi|^2}\right) \\
&\lesim \frac{1}{\sqrt{1+|\xi|}}||\widehat{\nabla u}(\xi,\cdot,t)||_{L^2_y}^2\left(\frac{1}{|\xi|^2}\right)\\
&\lesim \frac{1}{|\xi|^2\sqrt{1+|\xi|}}||\widehat{{\mathbb D}u}(\xi,\cdot,t)||_{L^2_y}^2.
\end{split}
\end{equation}
Then in order to satisfy \eqref{eq:betarec3}, we can take $\beta\le c|\xi|^2\sqrt{1+|\xi|}$ for a  sufficiently small constant $c$. Note that $\mu(\xi)\lesim (1+|\xi|)^3$ implies that the condition
\begin{equation}
\beta\le c\frac{|\xi|^2\mu(\xi)}{(1+|\xi|)^3} 
\end{equation}
is stronger than the other two conditions, so we can simply take 
\begin{equation}
\beta=c\frac{|\xi|^2\mu(\xi)}{(1+|\xi|)^3}. 
\end{equation}
in order to guarantee that \eqref{eq:betarec1}--\eqref{eq:betarec3} are satisfied.

With this choice of $\beta$ in \eqref{eq:fullenergydissipation}, we deduce that 
\begin{equation}
\mathcal E_\xi(t) \lesim \half||\hat u(\xi,\cdot,t)||_{L^2_y}^2+\frac{\mu(\xi)}{2}|\hat h(\xi,t)|^2+\beta(\hat{u}(\xi,\cdot,t),\hat{v}(\xi,\cdot,t))_{L^2_y}\lesim\mathcal E_\xi(t)
\end{equation}
and 
\begin{equation}
\begin{split}
\half||\widehat{{\mathbb D}u}(\xi,\cdot,t)||_{L^2_y}^2 
& -\beta\hat u_N(\xi,\ell,t)(\hat u(\xi,\cdot,t),w)_{L^2_y}  +\frac{\beta}{2}(\widehat{{\mathbb D}u}(\xi,\cdot,t),\widehat{{\mathbb D}v}(\xi,\cdot,t))_{L^2_y}+\beta\mu(\xi) |\hat h(\xi,t)|^2 \\
&\gesim||\widehat{{\mathbb D}u}(\xi,\cdot,t)||_{L^2_y}^2+\beta\mu(\xi) |\hat h(\xi,t)|^2\\
&\gesim ||\widehat{\nabla u}(\xi,\cdot,t)||_{L^2_y}^2+\beta\mu(\xi) |\hat h(\xi,t)|^2\\
&\gesim |\xi|^2||\hat{u}(\xi,\cdot,t)||_{L^2_y}^2+\beta\mu(\xi) |\hat h(\xi,\cdot,t)|^2\\
&\gesim \beta\mathcal E_\xi(t)\\
&\gesim \beta\left(\half||\hat u(\xi,\cdot,t)||_{L^2_y}^2+\frac{\mu(\xi)}{2}|\hat h(\xi,t)|^2+\beta(\hat{u}(\xi,\cdot,t),\hat{v}(\xi,\cdot,t))_{L^2_y}\right),
\end{split}
\end{equation}
where the second inequality follows from \eqref{ineq:fourierkorn}, and the fourth one follows from the bounds on $\beta$ derived above.  With these inequalities in hand, we can view  \eqref{eq:fullenergydissipation} as a differential inequality, which we can integrate to see that 
\begin{multline}
\half||\hat u(\xi,\cdot,t)||_{L^2_y}^2+\frac{\mu(\xi)}{2}|\hat h(\xi,t)|^2+\beta(\hat{u}(\xi,\cdot,t),\hat{v}(\xi,\cdot,t))_{L^2_y}\\
\le \left(\half||\hat u(\xi,\cdot,0)||_{L^2_y}^2+\frac{\mu(\xi)}{2}|\hat h(\xi,0)|^2+\beta(\hat{u}(\xi,\cdot,0),\hat{v}(\xi,\cdot,0))_{L^2_y}\right)\exp\left(-c\frac{|\xi|^2\mu(\xi)}{(1+|\xi|)^3}t\right)
\end{multline}
for some $c>0$. In turn, the above inequalities imply that that
\begin{equation}
\mathcal E_\xi(t)\le b\mathcal E_\xi(0)\exp\left(-c\frac{|\xi|^2\mu(\xi)}{(1+|\xi|)^3}t\right)
\end{equation}
for constants $b,c>0$.  This completes the proof of \eqref{eq:decayrate}  when $\Sigma = \T^{N-1}$.

We now indicate how to modify the above proof to handle the case $\Sigma = \R^{N-1}$.  In this case we cannot use $V$ and $v$ as above since $x \mapsto e^{2\pi i x\cdot \xi}$ does not belong to $L^2(\R^{N-1})$.  To get around this, we consider a real-valued $\varphi \in C_c^\infty(\R^{N-1})$ and define  
\begin{equation}
V(x,y,t) = \int_{\R^{N-1}} \varphi(\xi) \hat{u}(\xi,y,t)e^{2\pi i \xi\cdot x} d\xi \text{ and } v(x,y,t) = \int_{\R^{N-1}} \varphi(\xi) \hat h(\xi,t)e^{2\pi i\xi\cdot x}w(\xi,y)e^{2\pi i \xi\cdot x} d\xi,
\end{equation}
where in the latter equation we have written $w(\xi,y)$ for the mapping defined by \eqref{eq:betawdef}.  We may then argue as above and employ Fubini's theorem to see that 
\begin{equation}\label{eq:beta_inf_1}
0=\int_{\R^{N-1}} \left[ \frac{d}{dt}\left(\half||\hat u(\xi,\cdot,t)||_{L^2_y}^2+\frac{\mu(\xi)}{2}|\hat h(\xi,\cdot,t)|^2\right)+\half||\widehat{{\mathbb D}u}(\xi,\cdot,t)||_{L^2_y}^2 \right] \varphi(\xi) d\xi
\end{equation}
and 
\begin{multline}\label{eq:beta_inf_2}
\int_{\R^{N-1}} \left[ \frac{d}{dt}(\hat{u}(\xi,\cdot,t),\hat{v}(\xi,\cdot,t))_{L^2_y}-(\hat{u}(\xi,\cdot,t),\partial_t \hat{v}(\xi,\cdot,t))_{L^2_y}\right] \varphi(\xi) d\xi \\
+ \int_{\R^{N-1}} \left[ \half(\widehat{{\mathbb D}u}(\xi,\cdot,t),\widehat{{\mathbb D}v}(\xi,\cdot,t))_{L^2_y}+\mu(\xi) |\hat h(\xi,t)|^2 \right] \varphi(\xi) d\xi =0.
\end{multline}
Since \eqref{eq:beta_inf_1} and \eqref{eq:beta_inf_2} hold for all such $\varphi$ we conclude that \eqref{eq:firsteq} and \eqref{eq:secondeq} hold for almost every $\xi \in \hat{\Sigma} = \R^{N-1}$.  With these in hand we may then employ the above argument for a.e. $\xi$ to see that \eqref{eq:decayrate} holds for a.e. $\xi$.

\end{proof}

\begin{remark}
 In the case $\Sigma = \T^{N-1}$ we use counting measure, so Theorem \ref{thm:xidecayrate} tells us that \eqref{eq:decayrate} actually holds for all $\xi \in \Z^{N-1} \backslash \{0\}$.
\end{remark}

When $\Sigma = \T^{N-1}$, Theorem \ref{thm:xidecayrate} tells us nothing about the behavior of $\mathcal{E}_\xi$ when $\xi =0$.  We handle this case now by showing that $\mathcal{E}_0$ decays exponentially.

\begin{thm}\label{thm:xidecayrate_zero}
Suppose that $u$ is a weak solution to $\eqref{eq:fullweak}$ with $\Sigma=\T^{N-1}$. Then there is a constant $\delta>0$ such that
\begin{equation}\label{eq:zeromodedecayrate}
\mathcal E_0(t)\le \mathcal E_0(0)\exp\left(-\delta t\right),
\end{equation}
where $\mathcal{E}_0$ is defined by \eqref{term:xienergy} with $\xi =0$.
\end{thm}
\begin{proof}
Recall that in our notion of solution for $\Sigma = \T^{N-1}$ we require that $h$ satisfies the zero-average condition \eqref{eq_zero_avg}.  On the Fourier side this corresponds to the condition $\hat{h}(0,t) =0$ for $t \ge 0$.  Then 
\begin{equation}
 \mathcal{E}_0(t)=\half \int_0^\ell|\hat u(0,y,t)|^2dy+\frac{\mu(0)}{2}|\hat h(0,t)|^2 =\half \int_0^\ell|\hat u(0,y,t)|^2dy.
\end{equation}
Thus \eqref{eq:firsteq} in the proof of the Theorem \ref{thm:xidecayrate} becomes
\begin{equation}\label{eq:expdecay0mode}
0=\frac{d}{dt}\left(\half||\hat u(0,\cdot,t)||_{L^2_y}^2\right)+\half||\widehat{{\mathbb D}u}(0,\cdot,t)||_{L^2_y}^2.
\end{equation}
From \eqref{ineq:fourierpoincare} and \eqref{ineq:fourierkorn},  we get that
\begin{equation}
||\hat u(0,\cdot,t)||_{L^2_y}^2\lesim ||\widehat{{\mathbb D}u}(0,\cdot,t)||_{L^2_y}^2 
\end{equation}
which, together with \eqref{eq:expdecay0mode}, gives us the exponential decay in \eqref{eq:zeromodedecayrate}.
\end{proof}

%%%%%%%%%%%%%%%%%%%%%%%%%%%%%%%%%%%%%%%%%%%%%%%%%
%%%%%%%%%%%%%%%%%%%%%%%%%%%%%%%%%%%%%%%%%%%%%%%%%
\subsection{Sharpness for high frequencies  }
%%%%%%%%%%%%%%%%%%%%%%%%%%%%%%%%%%%%%%%%%%%%%%%%% 
%%%%%%%%%%%%%%%%%%%%%%%%%%%%%%%%%%%%%%%%%%%%%%%%%

In this subsection, we demonstrate that the decay bounds given in $\eqref{thm:xidecayrate}$ are tight for large frequencies $\xi \in \hat{\Sigma}$.  In order to prove this, we must impose some growth constraints on $\mu$ that force the operator $M$ to be of order less than $3$.

\begin{thm}\label{thm:ansatzfullbigxi}
Suppose that 
\begin{equation}\label{eq:ansatzfull_assump}
 \lim_{\abs{\xi} \to \infty} \frac{\mu(\xi)}{\abs{\xi}^3} =0 \text{ and } 1 \ls \liminf_{\abs{\xi} \to \infty} \mu(\xi).
\end{equation}
Then for $\xi \in \hat{\Sigma}$ sufficiently large there is a solution $(u,h)$ to $\eqref{eq:fullstrong}$ such that 
\begin{equation}\label{eq:ansatzfull_lowerbound}
\mathcal{E}_\xi(t) \ge c_1 \exp\left(-c_2\frac{\mu(\xi)}{|\xi|}t\right)  \mathcal{E}_\xi(0)
\end{equation}
for constants $c_1,c_2>0$ independent of $\xi$.
\end{thm}
\begin{proof}
We take the Fourier transform of $\eqref{eq:fullstrong}$ and consider frequencies $\xi \in \hat{\Sigma} \backslash \{0\}$.  Write  
\begin{equation}
\hat v(\xi,y,t)=\hat u_N(\xi,y,t) \text{ and }\hat w(\xi,y,t)=-i(\hat u_1,\ldots,\hat u_{N-1})(\xi,y,t). 
\end{equation}
To construct a solution to \eqref{eq:fullstrong} we will find a transformed velocity $(\hat{w},\hat{v})$ and pressure $\hat{p}$ satisfying the PDEs
\begin{align}
-2\pi\xi\cdot\hat w(\xi,y,t)+\partial_y\hat v(\xi,y,t)&=0\label{eq:fullbulk1}\\
\partial_t\hat w(\xi,y,t) + 2\pi\xi\hat p(\xi,y,t)+4\pi^2|\xi|^2\hat w(\xi,y,t)-\partial_y^2\hat w(\xi,y,t)&=0\label{eq:fullbulk2}\\
\partial_t\hat v(\xi,y,t)+\partial_y\hat p(\xi,y,t)+4\pi^2|\xi|^2\hat v(\xi,y,t)-\partial_y^2\hat v(\xi,y,t)&=0\label{eq:fullbulk3}
\end{align}
for $t \ge 0$ and $y \in (0,\ell)$, together with the boundary conditions
\begin{align}
2\pi\xi\hat v(\xi,\ell,t)+\partial_y\hat w(\xi,\ell,t)&=0\label{eq:fulltop1}\\
\hat p(\xi,\ell,t)-2\partial_y\hat v(\xi,\ell,t)&=\mu(\xi)\hat h(\xi,t) \label{eq:fulltop2}\\
\partial_t\hat h(\xi,t)&=\hat v(\xi,\ell,t)\label{eq:fulltop3}
\end{align}
and  
\begin{align}
\hat v(\xi,0,t)&=0\label{eq:fullbot1}\\
\hat w(\xi,0,t)&=0.\label{eq:fullbot2}
\end{align}
The solution to the transformed problem will be of the form
\begin{align}
\hat v(\xi,y,t) &=e^{-\rho t}\sum_{j=1}^4 v_j e^{a_j y}, \;
\hat w(\xi,y,t) =\xi e^{-\rho t}\sum_{j=1}^4 w_j e^{a_j y},  \label{eq:sharpform_1} \\
\hat p(\xi,y,t)&=e^{-\rho t}\sum_{j=1}^4 p_j e^{a_j y}, \text{ and }
\hat h(\xi,t)=e^{-\rho t} h_0\label{eq:sharpform_2}.
\end{align}
Note that $v_j,w_j,p_j,a_j,h_0,\rho$ depend on $\xi$, but we suppress this in the notation for the sake of brevity in what follows.  We will show that we can construct nontrivial solutions of this form with
\begin{equation}\label{bound:rho}
0 <  \frac{1}{4\pi} \left(\frac{\mu(\xi)}{|\xi|}\right) \le  \rho\le \left(1+\frac{1}{4\pi} \right)\left(\frac{\mu(\xi)}{|\xi|}\right).
\end{equation}
Once this is established, the bound \eqref{eq:ansatzfull_lowerbound} follows immediately.  When $\Sigma = \T^{N-1}$ we may use $\hat{v},\hat{w},\hat{p}$, and $\hat{h}$ as a single Fourier mode solution to \eqref{eq:fullstrong}, but when $\Sigma = \R^{N-1}$ we again run into the problem that $x \mapsto e^{2\pi i x\cdot \xi}$ does not belong to $L^2(\R^{N-1})$.  To get around this we instead solve the transformed equations in an open set $U \subseteq \R^{N-1}$ such that $\xi \in U$, and we then use a Fourier synthesis weighted by a function $\varphi \in C_c^\infty(U)$ such that $\varphi =1$ in a ball centered at $\xi$.  The resulting pair $(u,h)$ is then a weak solution to \eqref{eq:fullstrong} satisfying \eqref{eq:ansatzfull_lowerbound}.

If $\hat v,\hat w,\hat p,\hat h$ are of the form \eqref{eq:sharpform_1}--\eqref{eq:sharpform_2}, then $\eqref{eq:fullbulk1},\eqref{eq:fullbulk2},\eqref{eq:fullbulk3}$
become
\begin{align}
-2\pi|\xi|^2w_j+a_jv_j&=0\\
(4\pi^2|\xi|^2-\rho-a_j^2)w_j+2\pi p_j&=0\\
(4\pi^2|\xi|^2-\rho-a_j^2)v_j+a_jp_j&=0.
\end{align}
Combining the last two equations, we get
\begin{equation}
0=(4\pi^2|\xi|^2-\rho-a_j^2)a_jw_j-2\pi(4\pi^2|\xi|^2-\rho-a_j^2)v_j=(4\pi^2|\xi|^2-\rho-a_j^2)(a_jw_j-2\pi v_j),
\end{equation}
and upon substituting in the first equation we get
\begin{equation}
0=(4\pi^2|\xi|^2-\rho-a_j^2)(a_j^2-4\pi^2|\xi|^2)v_j=0.
\end{equation}
Then we set
\begin{equation}
a_1=2\pi|\xi|,  
a_2=-2\pi|\xi|, 
a_3=\sqrt{4\pi^2|\xi|^2-\rho}, 
a_4=-\sqrt{4\pi^2|\xi|^2-\rho},
\end{equation}
which are well-defined and distinct provided $\rho\ne 4\pi^2|\xi|^2$.
We also let
\begin{equation}
w_j=\frac{a_j}{2\pi|\xi|^2}v_j \text{ and }
p_j=-\frac{(4\pi^2|\xi|^2-\rho-a_j^2)}{a_j}v_j.
\end{equation}
We will also define $\eell=2\pi\ell$. Then the conditions \eqref{eq:fullbot1},\eqref{eq:fullbot2} on the bottom become
\begin{align}
0&=v_1+v_2+v_3+v_4\label{eq:1.0}\\
0&=v_1-v_2+\frac{\sqrt{4\pi^2|\xi|^2-\rho}}{2\pi|\xi|}v_3-\frac{\sqrt{4\pi^2|\xi|^2-\rho}}{2\pi|\xi|}v_4,\label{eq:2.0}
\end{align}
where \eqref{eq:fullbot2} has been multiplied by $|\xi|$. Condition \eqref{eq:fulltop1} at the top becomes
\begin{align}
0&=\left(2\pi+2\pi\right) e^{2\pi|\xi|\ell}v_1+\left(2\pi+2\pi\right) e^{-2\pi|\xi|\ell}v_2\nonumber\\
&\qquad+\left(2\pi+\frac{4\pi^2|\xi|^2-\rho}{2\pi|\xi|^2}\right) e^{\sqrt{4\pi^2|\xi|^2-\rho}\ell}v_3+\left(2\pi+\frac{4\pi^2|\xi|^2-\rho}{2\pi|\xi|^2}\right) e^{-\sqrt{4\pi^2|\xi|^2-\rho}\ell}v_4\nonumber\\
&\Leftrightarrow  \nonumber\\
0&=e^{\eell|\xi|}v_1+e^{-\eell|\xi|}v_2\nonumber\\
&\qquad+\left(1-\frac{\rho}{8\pi^2|\xi|^2}\right) e^{\sqrt{|\xi|^2-\frac{\rho}{4\pi^2}}\eell}v_3+\left(1-\frac{\rho}{8\pi^2|\xi|^2}\right) e^{-\sqrt{|\xi|^2-\frac{\rho}{4\pi^2}}\eell}v_4.\label{eq:3.0}
\end{align}
The conditions \eqref{eq:fulltop2} and \eqref{eq:fulltop3} at the top become
\begin{align}
\mu(\xi)h_0&=\left(\frac{\rho}{2\pi|\xi|}-4\pi|\xi|\right)e^{2\pi|\xi|\ell}v_1-\left(\frac{\rho}{2\pi|\xi|}-4\pi|\xi|\right)e^{-2\pi|\xi|\ell}v_2\nonumber\\
&\qquad-2\sqrt{4\pi^2|\xi|^2-\rho}e^{\sqrt{4\pi^2|\xi|^2-\rho}\ell}v_3+2\sqrt{4\pi^2|\xi|^2-\rho}e^{-\sqrt{4\pi^2|\xi|^2-\rho}\ell}v_4\label{eq:semi4}\\
0&=\rho h_0+e^{\eell|\xi|}v_1+e^{-\eell|\xi|}v_2+e^{\sqrt{|\xi|^2-\frac{\rho}{4\pi^2}}\eell}v_3+e^{-\sqrt{|\xi|^2-\frac{\rho}{4\pi^2}}\eell}v_4.\label{eq:hemi4}
\end{align}
Substituting \eqref{eq:semi4} into \eqref{eq:hemi4}, we get
\begin{align}
0&=\left(\frac{\frac{\rho}{2\pi|\xi|}-4\pi|\xi|}{\mu(\xi)}\rho+1\right)e^{\eell|\xi|}v_1+\left(\frac{-\frac{\rho}{2\pi|\xi|}+4\pi|\xi|}{\mu(\xi)}\rho+1\right)e^{-\eell|\xi|}v_2\nonumber\\
&\qquad +\left(\frac{-2\sqrt{4\pi^2|\xi|^2-\rho}}{\mu(\xi)}\rho+1\right)e^{\sqrt{|\xi|^2-\frac{\rho}{4\pi^2}}\eell}v_3+\left(\frac{2\sqrt{4\pi^2|\xi|^2-\rho}}{\mu(\xi)}\rho+1\right)e^{-\sqrt{|\xi|^2-\frac{\rho}{4\pi^2}}\eell}v_4.\label{eq:4.0}
\end{align}
Therefore, we have a solution if and only if $v_1,v_2,v_3,v_4$ satisfy $\eqref{eq:1.0}, \eqref{eq:2.0}, \eqref{eq:3.0}, \eqref{eq:4.0}$. This means we want to find a nonzero solution to
\begin{equation}\label{eq:matrixproduct0}
A(\rho,\xi)
\begin{pmatrix}
v_1\\
v_2\\
v_3\\
v_4
\end{pmatrix}=0
\end{equation}
where $A(\rho,\xi) \in \R^{4 \times 4}$ is given by
\begin{equation}
A(\rho,\xi)=\begin{pmatrix}\label{eq:expressionforA}
1 & 1 & 1 & 1\\
1 & -1 & \Gamma_{23} & -\Gamma_{23}\\
e^{\eell|\xi|} & e^{-\eell|\xi|} &  \Gamma_{33}e^{\eell \sqrt{|\xi|^2-\frac{\rho}{4\pi^2}}} & \Gamma_{33} e^{-\eell\sqrt{|\xi|^2-\frac{\rho}{4\pi^2}}}\\
\Gamma_{41}e^{\eell|\xi|} & \Gamma_{42}e^{-\eell|\xi|} &\Gamma_{43}e^{\eell \sqrt{|\xi|^2-\frac{\rho}{4\pi^2}}} & \Gamma_{44}e^{-\eell\sqrt{|\xi|^2-\frac{\rho}{4\pi^2}} }
\end{pmatrix}
\end{equation}
for 
\begin{equation}
\begin{split}
\Gamma_{23}&=\frac{\sqrt{4\pi^2|\xi|^2-\rho}}{2\pi|\xi|}, \; \Gamma_{33}=\left(1-\frac{\rho}{8\pi^2|\xi|^2}\right),\\
\Gamma_{41}&=\frac{\frac{\rho}{2\pi|\xi|}-4\pi|\xi|}{\mu(\xi)}\rho+1, \;
\Gamma_{42}=\frac{-\frac{\rho}{2\pi|\xi|}+4\pi|\xi|}{\mu(\xi)}\rho+1, \\
\Gamma_{43}&=\frac{-2\sqrt{4\pi^2|\xi|^2-\rho}}{\mu(\xi)}\rho+1, \;
\Gamma_{44}=\frac{2\sqrt{4\pi^2|\xi|^2-\rho}}{\mu(\xi)}\rho+1.
\end{split}
\end{equation}
Thus there is a solution to the equations with a given $\rho,\xi$ if and only if $\det A(\rho,\xi)=0$.

In light of \eqref{eq:ansatzfull_assump}, we may assume that $\abs{\xi}$ is large enough that 
\begin{equation}
 \left(1 + \frac{1}{4\pi}\right) \frac{\mu(\xi)}{\abs{\xi}^3} <1.
\end{equation}
We then define $\kappa_\pm >0$ via
\begin{equation}
 \kappa_\pm = \frac{1+4\pi \pm 1}{8\pi}. 
\end{equation}
Choosing $\rho_\pm=\kappa_\pm \frac{\mu(\xi)}{|\xi|}$, we then have that  
\begin{equation}
0 <  \rho_\pm \le \rho_+ = \left(1 + \frac{1}{4\pi}\right) \frac{\mu(\xi)}{\abs{\xi}} < \abs{\xi}^2.
\end{equation}
Note that, by \eqref{bound:rho}, with $\rho = \rho_\pm$ we have 
\begin{equation}
|\Gamma_{23}|,|\Gamma_{33}|,|\Gamma_{41}|,|\Gamma_{42}|,|\Gamma_{43}|,|\Gamma_{44}|\le |\xi|^{10},
\end{equation}
so if we write out $\det A(\rho,\xi)$ as a polynomial of the entries of $A(\rho,\xi)$ and use the fact that $0 <\rho = \rho_\pm <|\xi|^2$, we get
\begin{equation}\label{eq:A_determinant}
\det A(\rho,\xi)=(\Gamma_{23}-1)\left(\Gamma_{43}-\Gamma_{33}\Gamma_{41}\right)e^{\eell|\xi|+\eell \sqrt{|\xi|^2-\frac{\rho}{4\pi^2}}}+O(|\xi|^{40} e^{\eell|\xi|})+O(|\xi|^{40} e^{\eell\sqrt{|\xi|^2-\frac{\rho}{4\pi^2}}}).
\end{equation}

Next note that  
\begin{equation}\label{eq:gamma_bound_1}
\Gamma_{23}-1<0 \text{ and }|\Gamma_{23}-1| = 1 - \sqrt{1 - \frac{\kappa_\pm \mu(\xi)}{4\pi \abs{\xi}^3} } \le
\frac{\kappa_\pm \mu(\xi)}{4\pi \abs{\xi}^3} 
\end{equation}
for $\abs{\xi}$ sufficiently large.  We also have
\begin{equation}\label{eq:gamma_bound_2}
\begin{split} 
\Gamma_{43}-\Gamma_{33}\Gamma_{41}&=\frac{-2\kappa_\pm}{|\xi|}\sqrt{4\pi^2|\xi|^2-\kappa_\pm\frac{\mu(\xi)}{|\xi|}}+1-\left(1-\frac{\kappa_\pm\mu(\xi)}{8\pi^2|\xi|^3}\right)\left(1-4\pi\kappa_\pm+\frac{\kappa_\pm^2\mu(\xi)}{2\pi|\xi|^3}\right)\\
&=-4\pi\kappa_\pm+\frac{4\pi\kappa_\pm\mu(\xi)}{8\pi^2|\xi|^3}+1-(1-4\pi\kappa_\pm)+\left(\frac{\kappa_\pm(1-4\pi\kappa_\pm)}{8\pi^2}-\frac{\kappa_\pm^2}{2\pi}\right)\frac{\mu(\xi)}{|\xi|^3}+O\left(\frac{\mu(\xi)^2}{|\xi|^6}\right)\\
&=\frac{(1+4\pi)\kappa_\pm-8\pi\kappa_\pm^2}{8\pi^2}\frac{\mu(\xi)}{|\xi|^3}+O\left(\frac{\mu(\xi)^2}{|\xi|^6}\right)\\
&=(1+4\pi-8\pi\kappa_\pm)\frac{\kappa_\pm}{8\pi^2}\frac{\mu(\xi)}{|\xi|^3}+O\left(\frac{\mu(\xi)^2}{|\xi|^6}\right) \\
&= \mp  \frac{\kappa_\pm}{8\pi^2}\frac{\mu(\xi)}{|\xi|^3}+O\left(\frac{\mu(\xi)^2}{|\xi|^6}\right),
\end{split}
\end{equation}
where the last equality follows from the choice of $\kappa_\pm$.  Comparing \eqref{eq:A_determinant}--\eqref{eq:gamma_bound_2} and using \eqref{eq:ansatzfull_assump} then shows that for sufficiently large $\xi$
\begin{equation}
 \det A(\rho_+,\xi) >0 \text{ and } \det A(\rho_-,\xi) < 0.
\end{equation}
Consequently, the intermediate value theorem provides us with $\rho \in (\rho_-, \rho_+)$ such that $\det A(\rho,\xi) =0$, which in turn provides us with the desired solution.

\end{proof}

%%%%%%%%%%%%%%%%%%%%%%%%%%%%%%%%%%%%%%%%%%%%%%%%%
%%%%%%%%%%%%%%%%%%%%%%%%%%%%%%%%%%%%%%%%%%%%%%%%%
\subsection{Sharpness for low frequencies in the non-periodic case  }
%%%%%%%%%%%%%%%%%%%%%%%%%%%%%%%%%%%%%%%%%%%%%%%%% 
%%%%%%%%%%%%%%%%%%%%%%%%%%%%%%%%%%%%%%%%%%%%%%%%%

Theorem \ref{thm:ansatzfullbigxi} tells us that the decay bounds of Theorem \ref{thm:xidecayrate} are tight for sufficiently high frequencies under reasonable restrictions on $\mu$.  In the non-periodic case, $\Sigma =\R^{N-1}$, this leaves open the question of tightness for low frequencies.  We address this now.

\begin{thm}\label{thm:ansatzfullsmallxi}
Suppose that 
\begin{equation}\label{eq:ansatz_small_assump}
 0 < g \neq \frac{3}{\ell^3} \text{ and } \lim_{\xi \to 0} (\mu(\xi) - g) = 0.
\end{equation}
Then for $\xi$ sufficiently small, there is a solution $(u,h)$ to $\eqref{eq:fullstrong}$ on $\Sigma=\R^{N-1}$ such that 
\begin{equation}\label{eq:ansatz_small_decay}
\mathcal{E}_\xi(t) \ge b_1 e^{-b_2|\xi|^2t} \mathcal{E}_\xi(0) 
\end{equation}
for constants $b_1,b_2>0$ independent of $\xi$.
\end{thm}
\begin{proof}
We argue as in the proof of Theorem \ref{thm:ansatzfullbigxi}, reducing to proving \eqref{eq:matrixproduct0} with $A(\rho,\xi)$ given by \eqref{eq:expressionforA} as the criterion for having a solution satisfying \eqref{eq:ansatz_small_decay}.

To verify the criterion we will introduce another free parameter $\kappa \in\C\wo\{1\}$ and then set $\rho=4\pi^2\kappa |\xi|^2$.  We insist that $\kappa \neq 1$ in order to guarantee that $a_3\ne a_4$.  In light of \eqref{eq:ansatz_small_assump}, we may rewrite the terms in $A(\rho,\xi)$  as
\begin{equation}
\begin{split}
\Gamma_{23}&=\frac{\sqrt{4\pi^2|\xi|^2-\rho}}{2\pi|\xi|}=\sqrt{1-\kappa}\\
\Gamma_{33}&=\left(1-\frac{\rho}{8\pi^2|\xi|^2}\right)=1-\frac{\kappa}{2}\\
\Gamma_{41}&=\frac{\frac{\rho}{2\pi|\xi|}-4\pi|\xi|}{\mu(\xi)}\rho+1=1+\frac{8\pi^3\kappa(\kappa-2)|\xi|^3}{g}+o(|\xi|^3)\\
\Gamma_{42}&=\frac{-\frac{\rho}{2\pi|\xi|}+4\pi|\xi|}{\mu(\xi)}\rho+1=1-\frac{8\pi^3\kappa(\kappa-2)|\xi|^3}{g}+o(|\xi|^3)\\
\Gamma_{43}&=\frac{-2\sqrt{4\pi^2|\xi|^2-\rho}}{\mu(\xi)}\rho+1=1-\frac{16\pi^3\kappa\sqrt{1-\kappa}|\xi|^3}{g}+o(|\xi|^3)\\
\Gamma_{44}&=\frac{2\sqrt{4\pi^2|\xi|^2-\rho}}{\mu(\xi)}\rho+1=1+\frac{16\pi^3\kappa\sqrt{1-\kappa}|\xi|^3}{g}-o(|\xi|^3),
\end{split}
\end{equation}
where here we write $o(\abs{\xi}^3)$ to mean a quantity vanishing faster than $\abs{\xi}^3$ as $\xi \to 0$, with the parameter $\kappa$ viewed as fixed.  Then 
\begin{equation}
A(\rho,\xi)=\begin{pmatrix}\label{eq:secondexpressionforA}
    1 & 1 & 1 & 1\\
		1 & -1 & \Gamma_{23} & -\Gamma_{23}\\
		e^{\eell|\xi|} & e^{-\eell|\xi|} &  \Gamma_{33}e^{\sqrt{1-\kappa}\eell|\xi|} & \Gamma_{33} e^{-\sqrt{1-\kappa}\eell|\xi|}\\
		\Gamma_{41}e^{\eell|\xi|} & \Gamma_{42}e^{-\eell|\xi|} &\Gamma_{43}e^{\sqrt{1-\kappa}\eell|\xi|} & \Gamma_{44}e^{-\sqrt{1-\kappa}\eell|\xi|}
\end{pmatrix}.
\end{equation}

We now refer to the columns of the matrix $A(\rho,\xi)$ as $c_1,c_2,c_3,c_4$. Let $\tilde{A}(\rho,\xi) \in \C^{4 \times 4}$ be the matrix with columns
\begin{equation}
 \frac{c_1+c_2}{2},\frac{c_3+c_4}{2}, \frac{c_1-c_2}{2}, \text{ and }\frac{c_3-c_4}{2}.
\end{equation}
The matrix $\tilde A$ can be obtained from $A$ by means of column operations, so $\det A(\rho,\xi)=0$ if and only if $\det \tilde{A}(\rho,\xi)=0$.  Then
\begin{equation}
\tilde{A}(\rho,\xi)=
\begin{pmatrix}
    1 & 1 & 0 & 0\\
		0 & 0 & 1 & \sqrt{1-\kappa}\\
		1+O(|\xi|^2) & 1-\frac\kappa 2 +O(|\xi|^2) & \lambda|\xi|+O(|\xi|^3) & \lambda(1-\frac\kappa 2)\sqrt{1-\kappa}|\xi|+O(|\xi|^3)\\
		1+\frac{\lambda^2}{2}|\xi|^2+o(|\xi|^3) & 
		1+\frac{\lambda^2(1-\kappa)}{2}|\xi|^2+o(|\xi|^3) & 
		\Delta_{43}+o(|\xi|^3) & 
		\Delta_{44}+o(|\xi|^3)
\end{pmatrix}
\end{equation}
for
\begin{equation}
\begin{split}
\Delta_{43}&=\lambda|\xi|+\left(\frac{\lambda^3}{6}+\frac{8\pi^3\kappa(\kappa-2)}{g}\right)|\xi|^3\\
\Delta_{44}&=\lambda\sqrt{1-\kappa}|\xi|+\left(\frac{\lambda^3(1-\kappa)^{3/2}}{6}-\frac{16\pi^3\kappa\sqrt{1-\kappa}}{g}\right)|\xi|^3,
\end{split}
\end{equation}
where again in $O(\cdot)$ and $o(\cdot)$ we refer to quantities with $\kappa$ fixed and $\xi \to 0$.

We will now perform a sequence of row and column operations on $\tilde{A}(\rho,\xi)$, which do not change the vanishing of the determinant. Subtracting the first row from the third,  $1+\frac{\lambda^2}{2}|\xi|^2$ times the first row from the fourth, and $\lambda|\xi|$ times the second row from the fourth, we get
\begin{equation}
\begin{pmatrix}
    1 & 1 & 0 & 0\\
		0 & 0 & 1 & \sqrt{1-\kappa}\\
		O(|\xi|^2) & -\frac\kappa 2 +O(|\xi|^2) & \lambda|\xi|+O(|\xi|^3) & \lambda(1-\frac\kappa 2)\sqrt{1-\kappa}|\xi|+O(|\xi|^3)\\
		o(|\xi|^3) & 
		\frac{-\kappa\lambda^2}{2}|\xi|^2+o(|\xi|^3) & 
		\left(\frac{\lambda^3}{6}+\frac{8\pi^3\kappa(\kappa-2)}{g}\right)|\xi|^3+o(|\xi|^3)& 
		\left(\frac{\lambda^3(1-\kappa)^{3/2}}{6}-\frac{16\pi^3\kappa\sqrt{1-\kappa}}{g}\right)|\xi|^3+o(|\xi|^3)
\end{pmatrix}.
\end{equation}
Subtracting the first column from the second and $\sqrt{1-\kappa}$ times the third column from the fourth,  we get
\begin{equation}
\begin{pmatrix}
    1 & 0 & 0 & 0\\
		0 & 0 & 1 & 0\\
		O(|\xi|^2) & -\frac\kappa 2 +O(|\xi|^2) & O(|\xi|) & -\lambda(\frac\kappa 2)\sqrt{1-\kappa}|\xi|+O(|\xi|^3)\\
		o(|\xi|^3)& 
		\frac{-\kappa\lambda^2}{2}|\xi|^2+o(|\xi|^3)& 
		O(|\xi|^3)& 
		\left(\frac{-\lambda^3\kappa\sqrt{1-\kappa}}{6}-\frac{8\pi^3\kappa^2\sqrt{1-\kappa}}{g}\right)|\xi|^3+o(|\xi|^3)
\end{pmatrix}.
\end{equation}
The vanishing of the determinant of this matrix, and hence of $\det A(\rho,\xi)$, is then equivalent to 
\begin{equation}
0=\left(-\lambda\frac{\kappa}{2}\sqrt{1-\kappa}|\xi|\right)\left(-\frac{\kappa\lambda^2}{2}|\xi|^2\right)-\left(\frac{-\kappa}{2}\right)\left(\frac{-\lambda^3\kappa\sqrt{1-\kappa}}{6}-\frac{8\pi^3\kappa^2\sqrt{1-\kappa}}{g}\right)|\xi|^3+o(|\xi|^3),
\end{equation}
which in turn is equivalent to 
\begin{multline}
 0=\kappa^2\sqrt{1-\kappa}|\xi|^3\left(\frac{\lambda^3}{4}+\frac{-\lambda^3}{12}-\frac{4\pi^3\kappa}{g}+o(1)\right) =\kappa^2\sqrt{1-\kappa}|\xi|^3\left(\frac{\lambda^3}{6}-\frac{4\pi^3\kappa}{g}+o(1)\right) \\
 =: \kappa^2\sqrt{1-\kappa}|\xi|^3 \Psi(\kappa,\xi).
\end{multline}
Again we note that the term $o(1)$ depends on $\kappa$, so we cannot directly set the term in parentheses to $0$ to choose $\kappa$.  However, this expression tells us how to choose the value of $\kappa$ when $\xi =0$.  Indeed, 
\begin{equation}
0 =  \Psi(\kappa,0) = \frac{\lambda^3}{6}-\frac{4\pi^3\kappa}{g} \Leftrightarrow \kappa=\frac{g\lambda^3}{24\pi^3}.
\end{equation}
We have now shown that when $\xi = 0$, $\mu(0) =g$, and $\kappa=\frac{g\lambda^3}{24\pi^3}$ we have that $\det A(\rho,\xi) =0$.

Let us now consider $\mu \in (0,\infty)$ as a parameter rather than as the function of $\xi$.  The formula for $A(\rho,\xi)$ shows that the map 
\begin{equation}
 (\xi,\mu,\kappa) \mapsto \det A(\rho,\xi)
\end{equation}
is $C^1$ in an open set containing the zero 
\begin{equation}
 \xi = 0, \mu = g, \kappa = \frac{g\lambda^3}{24\pi^3}.
\end{equation}
Moreover, the derivative of this map with respect to $\kappa$ at the zero is nontrivial, so we may apply the implicit function theorem to solve for $\kappa = \kappa(\xi,\mu)$ in a neighborhood of $(0,g)$ such that $\det A(\rho,\xi)=0$.  Choosing $\kappa(\xi,\mu(\xi))$ then provides the desired solutions with 
\begin{equation}
 \kappa=\frac{g\lambda^3}{24\pi^3}+o(1) \text{ as }\xi \to 0.
\end{equation}
Note that if $\frac{g\lambda^3}{24\pi^3}>1$, then $\sqrt{1-\kappa}$ is not real, so $\kappa$ may not be real. This is why we need the implicit function theorem, rather than just the intermediate value theorem as in the proof of Theorem \ref{thm:ansatzfullbigxi}.
\end{proof}

%%%%%%%%%%%%%%%%%%%%%%%%%%%%%%%%%%%%%%%%%%%%%%%%%
%%%%%%%%%%%%%%%%%%%%%%%%%%%%%%%%%%%%%%%%%%%%%%%%%
%%%%%%%%%%%%%%%%%%%%%%%%%%%%%%%%%%%%%%%%%%%%%%%%%
\section{The periodic problem }
%%%%%%%%%%%%%%%%%%%%%%%%%%%%%%%%%%%%%%%%%%%%%%%%% 
%%%%%%%%%%%%%%%%%%%%%%%%%%%%%%%%%%%%%%%%%%%%%%%%%
%%%%%%%%%%%%%%%%%%%%%%%%%%%%%%%%%%%%%%%%%%%%%%%%%
 
Our goal in this section is to study the decay of the full energy $\mathcal{E}$, defined by \eqref{term:fullenergy}, in the periodic case, $\Sigma = \T^{N-1}$.  We will do this primarily within the context of the extra assumption that 
\begin{equation}\label{eq:mu_assump}
 \mu(\xi) = g+\sigma(2\pi |\xi|)^{2r}
\end{equation}
for $g,\sigma >0$.  When $r=1$, this is the same as linearized Navier-Stokes with a free boundary, gravity, and surface tension, while when $r=0$, this is the same as linearized Navier-Stokes with a free boundary, gravity, and no surface tension.

%%%%%%%%%%%%%%%%%%%%%%%%%%%%%%%%%%%%%%%%%%%%%%%%%
%%%%%%%%%%%%%%%%%%%%%%%%%%%%%%%%%%%%%%%%%%%%%%%%%
\subsection{Decay rates based on initial data  }
%%%%%%%%%%%%%%%%%%%%%%%%%%%%%%%%%%%%%%%%%%%%%%%%% 
%%%%%%%%%%%%%%%%%%%%%%%%%%%%%%%%%%%%%%%%%%%%%%%%%

We begin our analysis of the decay of $\mathcal{E}$ by considering the case when $r \ge 1/2$ in \eqref{eq:mu_assump}, in which case we can prove exponential decay.

\begin{thm}\label{thm:torusdecayexponential}
Suppose that $(u,h)$ is a weak solution of $\eqref{eq:fullweak}$ on $\Sigma=\T^{N-1}$ and that $\mu$ is of the form \eqref{eq:mu_assump} with $r\ge\half$.  If $u_0\in L^2(\Omega)$ and $h_0\in H^{r}(\Sigma)$, then there are constants $c,K>0$, depending on $g,\sigma,r$ such that 
\begin{equation}\label{eq:torusdecayexponential_est}
\mathcal{E}(t)\le K \mathcal{E}(0) e^{-ct},
\end{equation}
where the total energy $\mathcal{E}(t)$ as defined in \eqref{term:fullenergy}.
\end{thm}
\begin{proof}
First recall that since $\Sigma = \T^{N-1}$ we have that $\hat{\Sigma} = \Z^{N-1}$. For $\xi\ne 0$ the estimate \eqref{eq:decayrate} of Theorem \ref{thm:xidecayrate} tells us that 
\begin{equation}\label{eq:rbig}
\mathcal E_\xi(t)\le b\mathcal E_\xi(0)\exp\left(-ct\right),
\end{equation}
while for $\xi =0$ the bound  \eqref{eq:zeromodedecayrate} of Theorem \ref{thm:xidecayrate_zero} states that
\begin{equation}\label{eq:zerofreqdecay}
\mathcal{E}_0(t)\le \mathcal{E}_0(0)\exp\left(-\delta t\right).
\end{equation}

Recall that 
\begin{equation}
\mathcal{E}(t) = \sum_{\xi \in {\hat\Sigma}}\mathcal E_\xi(t). 
\end{equation}
This, \eqref{eq:rbig}, and \eqref{eq:zerofreqdecay} then imply that
\begin{equation}
\mathcal E(t)=\mathcal E_0(t)+\sum_{\xi\ne 0} \mathcal E_\xi(t)\le \mathcal E_0(0)e^{-\delta t} +b\sum_{\xi\ne 0}e^{-c t}\mathcal E_\xi(0)\lesim (b+1)\mathcal E(0)e^{-\min\{c,\delta\}t},
\end{equation}
which is \eqref{eq:torusdecayexponential_est}.
\end{proof}

\begin{remark}
By Theorem \ref{thm:torusdecayexponential}, we get exponential decay of the energy.  In fact, when $r>\half$, higher frequencies decay faster, so for $t>0$, we get that $u(x,y,t),h(x,y,t)$ are infinitely differentiable with respect to the horizontal variable $x$.
\end{remark}

We next turn our attention to the case in which $r < 1/2$ in \eqref{eq:mu_assump}.  In this case we do not get exponential decay, but instead algebraic decay directly linked to the regularity of the initial data.

\begin{thm}\label{thm:torusdecayalgebraic}
Suppose that $(u,h)$ is a weak solution of $\eqref{eq:fullweak}$ on $\Sigma=\T^{N-1}$ and that $\mu$ is of the form \eqref{eq:mu_assump} with $0 \le r <  \half$.  Then there exists a constant $K>0$ such that if $u_0 \in H^s_x(L^2_y)$ and $h_0\in H^{s+r}$ for $s > 0$, then 
\begin{equation}\label{eq:torusdecayalgebraic_est}
\mathcal E(t)\le K \left(||u_0||_{H^s_x(L^2_y)}^2+||h_0||_{H^{s+r}}^2\right) (1+t)^{\frac{-s}{\half-r}}. 
\end{equation}
\end{thm}
\begin{proof}
Theorems \ref{thm:xidecayrate} and \ref{thm:xidecayrate_zero} provide us with the estimate
\begin{equation}
\begin{split}
\mathcal E(t)&=\mathcal E_0(t)+\sum_{\xi\ne 0} \mathcal E_\xi(t)\\
&\le \mathcal E_0(0)e^{-\delta t} +b\sum_{\xi\ne 0} \left(|\xi|^{2s}\mathcal E_\xi(0)\right)|\xi|^{-2s}\exp\left(-c|\xi|^{2r-1}t\right)\\
&\le \mathcal E_0(0)e^{-\delta t} +b\exp(cc_0^{2r-1})\sum_{\xi\ne 0} \left(|\xi|^{2s}\mathcal E_\xi(0)\right)|\xi|^{-2s}\exp\left(-c|\xi|^{2r-1}(t+1)\right)\\
&\le \mathcal E_0(0)e^{-\delta t} +C\left(\frac{g}{2} ||u_0||_{H^s_x(L^2_y)}^2+\sigma(2\pi)^{2r} ||h_0||_{H^{s+r}}^2\right)\sup_{\xi\ne 0}|\xi|^{-2s}\exp\left(-c|\xi|^{2r-1}(t+1)\right).
\end{split}
\end{equation}
The substitution $a=|\xi|^{2r-1}(t+1)$ then gives us
\begin{equation}
\begin{split}
\mathcal E(t) &\le \mathcal E_0(0)e^{-\delta t} +\tilde C\left(||u_0||_{H^s_x(L^2_y)}^2+||h_0||_{H^{s+r}}^2\right)\sup_{a>0}a^{\frac{2s}{1-2r}}(t+1)^{\frac{-2s}{1-2r}}\exp\left(-ca\right)\\
&\le \mathcal E(0)e^{-\delta t}+\breve C\left(||u_0||_{H^s_x(L^2_y)}^2+||h_0||_{H^{s+r}}^2\right)(t+1)^{\frac{-s}{\half-r}}
\end{split}
\end{equation}
for some constant $\tilde{C}>0$. Then \eqref{eq:torusdecayalgebraic_est} follows immediately from this.
\end{proof}

%%%%%%%%%%%%%%%%%%%%%%%%%%%%%%%%%%%%%%%%%%%%%%%%%
%%%%%%%%%%%%%%%%%%%%%%%%%%%%%%%%%%%%%%%%%%%%%%%%%
\subsection{Understanding the transition near $r=1/2$  }
%%%%%%%%%%%%%%%%%%%%%%%%%%%%%%%%%%%%%%%%%%%%%%%%% 
%%%%%%%%%%%%%%%%%%%%%%%%%%%%%%%%%%%%%%%%%%%%%%%%%

Theorems \ref{thm:torusdecayexponential} and \ref{thm:torusdecayalgebraic} show that there is a transition in the decay behavior of the energy when $r=\half$ in \eqref{eq:mu_assump}, that is when $\mu(\xi)=g+2\pi\sigma |\xi|$.  We now examine this transition more carefully.

\begin{thm}\label{thm:transitiondecay}
Suppose that $(u,h)$ is a weak solution of $\eqref{eq:fullweak}$ on $\Sigma=\T^{N-1}$ and that $u_0\in H^s, h_0\in H^{s+3}$ for some $s>0$.  Then the following hold.
\begin{enumerate}
 \item If there exists $\alpha >0$ such that  
\begin{equation}
\mu(\xi)=g + 2\pi \sigma \frac{|\xi|}{(\log|\xi|)^\alpha} \text{ for } \abs{\xi} > 1,
\end{equation}
then we have that
\begin{equation}\label{eq:transitiondecay_log}
\mathcal E(t)\le K \left( ||u||_{H^s}^2+||h||_{H^{s+3}}^2 \right) \exp\left(-\breve C_st^{\frac{1}{1+\alpha}}\right),
\end{equation}
where the constants $\breve C_s,K$ depend on $s,\alpha$.

 \item If there exists $\alpha >0$ such that  
\begin{equation}
 \mu(\xi)=g + 2\pi \sigma \frac{|\xi|}{(\log\log|\xi|)^\alpha}  \text{ for } \abs{\xi} > e,
\end{equation}
then we have that
\begin{equation}\label{eq:transitiondecay_loglog}
\mathcal E(t)\le K \left(||u||_{H^s}^2+||h||_{H^{s+3}}^2 \right) \exp\left(-\breve C_s\frac{t}{(\log t)^{\alpha}}\right)
\end{equation}
where the constants $\breve C_s,K$ depend on $s,\alpha$. 
\end{enumerate}

\end{thm}
\begin{proof}
Throughout the proof we will write $\gamma : (e,\infty) \to (0,\infty)$ to refer to either 
\begin{equation}
 \gamma(z) = \frac{z}{(\log z)^\alpha} \text{ or } \gamma(z) =  \frac{z}{(\log \log z)^\alpha},
\end{equation}
which allows us to write 
\begin{equation}
 \mu(\xi) = g + 2\pi \sigma \gamma(\abs{\xi}) \text{ for } \abs{\xi} >e
\end{equation}
in both cases under consideration.  Note that in both cases $\gamma$ is such that $(e,\infty) \ni z \mapsto z/\gamma(z)$ is continuous and strictly increasing, and
\begin{equation}
\lim_{z\to\infty} \frac{z}{\gamma(z)}=\infty. 
\end{equation}
We also write $\beta :(0,\infty) \to (0,\infty)$ via either
\begin{equation}
 \beta(z) = \exp\left(z^{1/\alpha}\right) \text{ or }  \beta(z) = \exp\left(\exp\left(z^{1/\alpha}\right)\right)
\end{equation}
depending on the form of $\gamma$ above.  The value of $\beta$ is chosen so that $\beta(z/\gamma(z)) = z$, as is easily verified.

Theorems \ref{thm:xidecayrate} and \ref{thm:xidecayrate_zero} provide us with the estimates
\begin{equation}
\mathcal{E}_\xi(t)\le c_1\mathcal{E}_\xi(0)\exp\left(-\frac{c_2\mu(\xi)}{2\pi \sigma |\xi|}t\right) \text{ for }\xi \neq 0  \text{ and } \mathcal{E}_0(t) \le c_3 \mathcal{E}_0(0) e^{-c_4 t}
\end{equation}
for constants $c_1,c_2,c_3,c_4>0$.  Consequently, we can bound 
\begin{equation}
 \sum_{\abs{\xi} \le e} \mathcal{E}_\xi(t) \le c_3 e^{-c_5 t} \sum_{\abs{\xi} \le e} \mathcal{E}_\xi(0)
\end{equation}
for another constant $c_5 >0$.  Similarly,  we can bound
\begin{equation}
 \sum_{\abs{\xi} > e} \mathcal{E}_\xi(t) \le c_1 \exp\left(-\frac{c_2\gamma(\abs{\xi})}{ |\xi|}t\right)  \sum_{\abs{\xi} > e} \mathcal{E}_\xi(0).
\end{equation}
Thus
\begin{equation}
\begin{split}
\mathcal E(t)&=\sum_{|\xi|\le e}\mathcal E_\xi(t)+\sum_{|\xi|>e}\mathcal E_\xi(t)\\
&\lesim e^{-c_5 t}\sum_{|\xi|\le e}\mathcal E_\xi(0)+\sum_{|\xi|>e}\left(|\xi|^{2s}||\hat u(\xi)||_{L^2_y}^2+|\xi|^{2s+6}||\hat h(\xi)||^2\right)|\xi|^{-2s}\exp\left(-\frac{c_2\gamma(|\xi|)}{|\xi|}t\right)\\
&\lesim e^{-c_5 t}\sum_{|\xi|\le e}\mathcal E_\xi(0)+\left(||u||_{H^s}^2+||h||_{H^{s+3}}^2\right)\sup_{|\xi|>e}|\xi|^{-2s}\exp\left(-\frac{c_2\gamma(|\xi|)}{|\xi|}t\right).
\end{split}
\end{equation}

Next we use the substitution $a=\frac{\gamma(|\xi|)}{|\xi|}t >0$ to calculate
\begin{equation}
\begin{split}
\sup_{|\xi|>e}|\xi|^{-2s}\exp\left(-\frac{c_2\gamma(|\xi|)}{|\xi|}t\right)
&\le\sup_{a>0} (\beta(t/a))^{-2s}\exp(-c_2a)\\
&\le\sup_{a>0} \exp(-2s\log \beta(t/a)-c_2a)\\
&\le\exp(-\inf_{a>0}(2s\log \beta(t/a)+c_2a)).
\end{split}
\end{equation}
Thus, if we write 
\begin{equation}
 B = ||u||_{H^s}^2+||h||_{H^{s+3}}^2,
\end{equation}
then we have the bound
\begin{equation}\label{ineq:mina}
\mathcal E(t)\ls Be^{-c_5 t}+B\exp(-\inf_{a>0}(2s\log \beta(t/a)+c_2a)).
\end{equation}

For \eqref{ineq:mina} to be useful  we want a lower bound on $\inf_{a>0}(2s\log \beta(t/a)+c_2a)$. Let $a_*$ be such that $2s\log \beta(t/a_*)=c_2a_*$.  Since $2s\log \beta(t/a)$ is decreasing and $c_2a$ is increasing in $a$, we have that for all $a$,
\begin{equation}
\inf_{a>0}(2s\log \beta(t/a)+c_2a)\ge c_2a_*.
\end{equation}
Thus
\begin{equation}\label{bound:adecay}
\mathcal E(t)\le Be^{-c_5 t}+B\exp(-c_2a_*).
\end{equation}

We now handle the separate cases. For the first case 
\begin{equation}
\gamma(z)=\frac{z}{(\log z)^\alpha} \text{ and } \beta(z)=\exp(z^{1/\alpha}),
\end{equation}
and so
\begin{equation}
\begin{split}
&c_2a_* =2s\log \beta(t/a_*) \Leftrightarrow \frac{c_2}{2s}a_* =\left(\frac{t}{a_*}\right)^{\frac{1}{\alpha}}\\
&\Leftrightarrow C_sa_*^{\alpha} =\frac{t}{a_*}
\Leftrightarrow C_sa_*^{1+\alpha} =t 
 \Leftrightarrow a_* =\tilde C_st^{\frac{1}{1+\alpha}}
\end{split}
\end{equation}
where $C_s,\tilde C_s$ are constants depending on $s$. Then $\eqref{ineq:mina}$ gives us
\begin{equation}
\mathcal E(t)\ls B e^{-c_5 t}+B \exp(-c_2C_st^{\frac{1}{1+\alpha}})\le K B \exp\left(-\breve C_st^{\frac{1}{1+\alpha}}\right)
\end{equation}
for some constant $K>0$.  This proves \eqref{eq:transitiondecay_log}.

For the second case
\begin{equation}
\gamma(z)=\frac{z}{(\log \log z)^\alpha} \text{ and } \beta(z) = \exp \exp(z^{1/\alpha}), 
\end{equation}
and so
\begin{equation}
\begin{split}
&c_2a_*=2s\log \beta(t/a_*) \Leftrightarrow \frac{c_2}{2s}a_* =\exp\left(\left(\frac{t}{a_*}\right)^{\frac{1}{\alpha}}\right)\\
&\Leftrightarrow \log\left(C_sa_*\right)^{\alpha} =\frac{t}{a_*} \Leftrightarrow 
a_*\log\left(C_sa_*\right)^{\alpha} =t \Leftrightarrow
a_* \ge\tilde C_s\frac{t}{(\log t)^{\alpha}} 
\end{split}
\end{equation}
where $C_s,\tilde C_s$ are constants depending on $s$. Then $\eqref{ineq:mina}$ gives us
\begin{equation}
\mathcal E(t)\ls Be^{-c_4 t}+ B\exp\left(-c_2\tilde C_s\frac{t}{\log(t)^{\alpha}}\right)\le KB \exp\left(-\breve C_s\frac{t}{\log(t)^{\alpha}}\right)
\end{equation}
for a constant $K >0$.  This proves \eqref{eq:transitiondecay_loglog}.

\end{proof}

%%%%%%%%%%%%%%%%%%%%%%%%%%%%%%%%%%%%%%%%%%%%%%%%%
%%%%%%%%%%%%%%%%%%%%%%%%%%%%%%%%%%%%%%%%%%%%%%%%%
%%%%%%%%%%%%%%%%%%%%%%%%%%%%%%%%%%%%%%%%%%%%%%%%%
\section{The non-periodic problem }
%%%%%%%%%%%%%%%%%%%%%%%%%%%%%%%%%%%%%%%%%%%%%%%%% 
%%%%%%%%%%%%%%%%%%%%%%%%%%%%%%%%%%%%%%%%%%%%%%%%%
%%%%%%%%%%%%%%%%%%%%%%%%%%%%%%%%%%%%%%%%%%%%%%%%%

In this section we focus our attention on the decay properties of $\mathcal{E}$ when $\Sigma = \R^{N-1}$.  We will again assume that $\mu$ has the special structure given in \eqref{eq:mu_assump}.  The decay properties will depend heavily on the sort of assumptions we place on the initial data.  The reason for this is that the high and low frequencies decay at very different rates, with the low frequencies decaying more slowly.  The low frequencies may be handled with different arguments depending on which spaces the data belong to.

%%%%%%%%%%%%%%%%%%%%%%%%%%%%%%%%%%%%%%%%%%%%%%%%%
%%%%%%%%%%%%%%%%%%%%%%%%%%%%%%%%%%%%%%%%%%%%%%%%%
\subsection{Decay with data in $L^2$-based spaces }
%%%%%%%%%%%%%%%%%%%%%%%%%%%%%%%%%%%%%%%%%%%%%%%%% 
%%%%%%%%%%%%%%%%%%%%%%%%%%%%%%%%%%%%%%%%%%%%%%%%%

We begin our analysis when $\Sigma = \R^{N-1}$ by considering the case in which the initial data are assumed to belong to $L^2-$based spaces.  In this context we will follow \cite{GuoTice2013a2} and assume that  $||I_\lambda u||_{L^2(\Omega)}^2,||I_\lambda h||_{L^2(\Omega)}^2<\infty$, where
\begin{equation}
||I_\lambda u||_{L^2(\Omega)}^2:=\int_{\R^{N-1}\wo\{0\}} |\xi|^{-2\lambda}||\hat u(\xi)||_{L^2_y}^2 d\xi 
\end{equation}
and 
\begin{equation}
||I_\lambda h||_{L^2(\Sigma)}^2:=\int_{\R^{N-1}\wo\{0\}} |\xi|^{-2\lambda}|\hat h(\xi)|^2 d\xi .
\end{equation}
Control of these terms allows us to prove crucial estimates for the low frequency part of the solutions.

\begin{thm}\label{thm:nonper_L2_decay}
Suppose that $(u,h)$ is a weak solution of $\eqref{eq:fullweak}$ on $\Sigma=\R^{N-1}$ and that $\mu$ is of the form \eqref{eq:mu_assump} with $0 \le r \le 1$.  Assume that the initial data $(u_0,h_0$ satisfy  $||I_\lambda u_0||_{L^2(\Omega)},||I_\lambda h_0||_{L^2(\Sigma)}<\infty$, where these terms are as defined above.  Further assume that if $r\ge \half$, then  $u_0\in L^2$ and $h_0\in H^r$, and if $0\le r<\half$, then $u_0\in H^{\lambda(\half-r)}_x(L^2_y)$ and $h_0\in H^{\frac{\lambda}{2}+r(1-\lambda)}$.   Then
\begin{equation}\label{eq:nonper_L2_decay_est}
\mathcal E(t)\le K(1+t)^{-\lambda}
\end{equation}
for some $K$ which depends on the initial data.
\end{thm}
\begin{proof}
Fix any $c_0>0$, which we use as the cutoff between low and high frequencies.  Then we can write $\mathcal E(t)=\mathcal E_{<c_0}(t)+\mathcal E_{>c_0}(t)$ for
\begin{equation}
\mathcal E_{<c_0}(t) =\int_{0<|\xi|<c_0}\mathcal E_\xi(t) d\xi \text{ and }
\mathcal E_{>c_0}(t) =\int_{|\xi|>c_0}\mathcal E_\xi(t) d\xi.
\end{equation}
We already know how $\mathcal E_{>c_0}(t)$ decays.  Indeed, the Sobolev spaces chosen in the hypotheses allow us to employ the same arguments used in Theorems \ref{thm:torusdecayexponential} (for $r \ge 1/2$) and \ref{thm:torusdecayalgebraic} (for $0 \le r < 1/2$), with the sums replaced by integrals, to see that 
\begin{equation}\label{eq:nonper_L2_decay_1}
\mathcal E_{>c_0}(t)\le K(1+t)^{-\lambda}
\end{equation}
for a $K$ depending on the data.  

To understand the behavior of $\mathcal E_{<c_0}(t)$, we use the estimate \eqref{eq:decayrate} of Theorem \ref{thm:xidecayrate} to see that
\begin{equation}
\mathcal E_\xi(t)\le b\mathcal E_\xi(0)\exp\left(-c|\xi|^2t\right) \text{ for } \abs{\xi} < c_0,
\end{equation}
which implies that 
\begin{equation}
\begin{split}
\mathcal E_{<c_0}(t)&\le \int_{\{|\xi|<c_0\}} b\mathcal E_\xi(0)\exp\left(-c|\xi|^2t\right) d\xi \le\int_{\{|\xi|<c_0\}} be^{cc_0^2}\mathcal E_\xi(0)\exp\left(-c|\xi|^2(t+1)\right) d\xi\\
&\le be^{cc_0^2}\left(\int_{\{|\xi|<c_0\}} \mathcal E_\xi(0) |\xi|^{-2\lambda}d\xi\right)\sup_{\xi} |\xi|^{2\lambda}\exp\left(-c|\xi|^2(t+1)\right).
\end{split}
\end{equation}
The substitution $a=|\xi|^2(t+1)$ then gives us
\begin{equation}\label{eq:nonper_L2_decay_2}
\begin{split}
\mathcal E_{<c_0}(t)&\le be^{cc_0^2}\left(||I_\lambda u_0||_{L^2(\Omega)}+(g+(2\pi c_0)^{2r})||I_\lambda h||_{L^2(\Sigma)}^2\right)\sup_{a>0} a^{\lambda}(t+1)^{-\lambda}\exp\left(-ca\right)\\
&\le\tilde K(1+t)^{-\lambda}
\end{split}
\end{equation}
for $\tilde{K}$ depending on the data.  Then \eqref{eq:nonper_L2_decay_est} follows by combining \eqref{eq:nonper_L2_decay_1} and \eqref{eq:nonper_L2_decay_2}.
\end{proof}

%%%%%%%%%%%%%%%%%%%%%%%%%%%%%%%%%%%%%%%%%%%%%%%%%
%%%%%%%%%%%%%%%%%%%%%%%%%%%%%%%%%%%%%%%%%%%%%%%%%
\subsection{Decay with data in $L^1$}
%%%%%%%%%%%%%%%%%%%%%%%%%%%%%%%%%%%%%%%%%%%%%%%%% 
%%%%%%%%%%%%%%%%%%%%%%%%%%%%%%%%%%%%%%%%%%%%%%%%%

We can also assume that the initial data is $L^1$, as is done in \cite{BealeNishida1985}, in which case we get the following variant of Theorem \ref{thm:nonper_L2_decay}.

\begin{thm}\label{thm:nonper_L1_decay}
Suppose that $(u,h)$ is a weak solution of $\eqref{eq:fullweak}$ on $\Sigma=\R^{N-1}$ and that $\mu$ is of the form \eqref{eq:mu_assump} with $0 \le r \le 1$.  Assume that the initial data satisfy  $u_0\in L^2_y(L^1_x),h_0\in L^1$.  Further assume that if $1/2 \le r \le 1$, then $u_0\in L^2$ and $h_0\in H^r$, and if  $0\le r<\half$, then  $u_0\in H^{\frac{N-1}{2}\left(\half-r\right)}_x(L^2_y)$ and $h_0\in H^{\frac{N-1}{4}+\frac{N+1}{2}r}$.   Then
\begin{equation}\label{eq:nonper_L1_decay_est}
\mathcal E(t)\le K(1+t)^{-\frac{N-1}{2}}
\end{equation}
for some $K$ that depends on the initial data.
\end{thm}
\begin{proof}
Fix some $c_0>0$. Then $\mathcal E(t)=\mathcal E_{<c_0}(t)+\mathcal E_{>c_0}(t)$, where
\begin{equation}
\mathcal E_{<c_0}(t) =\int_{0<|\xi|<c_0}\mathcal E_\xi(t)d\xi \text{ and }
\mathcal E_{>c_0}(t) =\int_{|\xi|>c_0}\mathcal E_\xi(t) d\xi. 
\end{equation}

The argument used in Theorem \ref{thm:nonper_L2_decay} may be readily adapted to show that
\begin{equation}\label{eq:nonper_L1_decay_1}
\mathcal E_{>c_0}(t)\le  K(1+t)^{-\frac{N-1}{2}}
\end{equation}
for a $K$ depending on the data.

To understand the behavior of $\mathcal E_{<c_0}(t)$, we note that the Fourier transform of an $L^1$ functions is $L^\infty$, so $\mathcal E_\xi(0)$ is bounded independently of $\xi$. Thus we get (performing the change of variables $a=\rho^2(t+1)$ when going to the fourth line)
\begin{equation}\label{eq:nonper_L1_decay_2}
\begin{split}
\mathcal E_{<c_0}(t)&\le \int_{\{|\xi|<c_0\}} b\mathcal E_\xi(0)\exp\left(-c|\xi|^2t\right)d\xi\\
                    &\le Ce^{cc_0^2}\int_{\{|\xi|<c_0\}} \exp\left(-c|\xi|^2(t+1)\right) d\xi\\
&\le \tilde C\int_0^{\infty} \rho^{N-2}\exp\left(-c\rho^2(t+1)\right) d\rho\\
&\le \tilde C\int_0^{\infty} (t+1)^{-\frac{N-2}{2}}a^{\frac{N-2}{2}}\exp\left(-ca\right)(2a(t+1))^{-\half})da\\
&\le 2\tilde C\left(\int_0^{\infty}a^{\frac{N-2}{2}-\half}\exp\left(-ca\right)\right)(t+1)^{-\frac{N-2}{2}-\half}\\
&\le K(t+1)^{-\frac{N-1}{2}}.
\end{split}
\end{equation}
Then \eqref{eq:nonper_L1_decay_est} follows by combining \eqref{eq:nonper_L1_decay_1} and \eqref{eq:nonper_L1_decay_2}.
\end{proof}

\appendix

%%%%%%%%%%%%%%%%%%%%%%%%%%%%%%%%%%%%%%%%%%%%%%%%%
%%%%%%%%%%%%%%%%%%%%%%%%%%%%%%%%%%%%%%%%%%%%%%%%%
%%%%%%%%%%%%%%%%%%%%%%%%%%%%%%%%%%%%%%%%%%%%%%%%%
\section{Some useful analytic facts }
%%%%%%%%%%%%%%%%%%%%%%%%%%%%%%%%%%%%%%%%%%%%%%%%% 
%%%%%%%%%%%%%%%%%%%%%%%%%%%%%%%%%%%%%%%%%%%%%%%%%
%%%%%%%%%%%%%%%%%%%%%%%%%%%%%%%%%%%%%%%%%%%%%%%%%

Here we have compiled some analytic facts that we use in the paper. The forms given here are the forms we use, and no attempt has been made to state them in any additional generality.  

First we record a simple version of the trace theorem.
\begin{thm}[Trace]\label{trace}
If $\Omega=(0,\ell)$ or $\Omega=\R^{N-1}\times (0,\ell)$ or $\Omega=\T^{N-1}\times (0,\ell)$, then there is a continuous linear map $\emph{Tr}:H^1(\Omega)\to L^2(\partial\Omega)$ so that for $f\in C^\infty(\bar\Omega)$, we have $\emph{Tr} f=f\upharpoonright \partial\Omega$ and so that $f, \emph{Tr} f$ satisfy the integration by parts formula for all $f\in H^1(\Omega)$.
\end{thm}
\begin{proof}
This is a special case of Theorem 3 in Chapter 5.9 of \cite{evans}.
\end{proof}

Next we record a version of the Poincar\'e inequality.

\begin{thm}[Poincar\'e inequality]\label{poincare}
If $\Omega=(0,\ell)$ or $\Omega=\R^{N-1}\times (0,\ell)$ or $\Omega=\T^{N-1}\times (0,\ell)$, then there is some constant $C>0$ so that for all $f\in H^1(\Omega)$ satisfying $f=0$ on $\{y=0\}$,
\begin{equation}
||f||_{H^1}\lesim ||Df||_{L^2}.
\end{equation}
\end{thm}
\begin{proof}
For $\R^{N-1}\times (0,\ell)$ and $\T^{N-1}\times (0,\ell)$, it follows from the Poincar\'e inequality for $(0,\ell)$, which in turn follows from integration and Minkowski's inequality.
\end{proof}

Next we record a version of Korn's inequality.

\begin{thm}[Korn's inequality]\label{korn}
If $\Omega=\R^{N-1}\times (0,\ell)$ or $\Omega=\T^{N-1}\times (0,\ell)$, then there is some constant $C>0$ so that for all $f\in H^1(\Omega)$ satisfying $f=0$ on $\{y=0\}$, we have
\begin{equation}
||Df||_{L^2}\lesim ||\mathbb Df||_{L^2}.
\end{equation}
\end{thm}
\begin{proof}
For a proof, see \cite{Beale1981}, Lemma 2.7. 
\end{proof}

Finally, we record a result about time derivatives.

\begin{thm}\label{thm:CL2}
If $v$ is a complex-valued function satisfying $v\in L^2_T(\Hxi(0,\ell))$ and $\partial_t v\in L^2_T(\Hxi^*(0,\ell))$, then $v\in C([0,T];L^2(0,\ell))$ and
\begin{equation}
\frac{d}{dt}||v||_{L^2(0,\ell)}^2=[\partial_t v,v]_{\Hxi^*,\Hxi}+\overline{[\partial_t v,v]_{(\Hxi)^*,\Hxi}}.
\end{equation}
The same holds if we replace $\Hxi$ with $\Hxihel$. Also, if $v$ is a complex-valued function satisfying $v\in L^2_T(\Hbot(\Omega))$ and $\partial_t v\in L^2_T(\Hbot^*(\Omega))$, then $v\in C([0,T];L^2(\Omega))$ and
\begin{equation}
\frac{d}{dt}||v||_{L^2(\Omega)}^2=[\partial_t v,v]_{(\Hbot)^*,\Hbot}+\overline{[\partial_t v,v]_{(\Hbot)^*,\Hbot}}.
\end{equation}
The same holds if we replace $\Hbot$ with $\Hhel$.
\end{thm}
\begin{proof}
This is proved in the same manner as Theorem 3 in Chapter 5.9 of \cite{evans}.
\end{proof}

%%%%%%%%%%%%%%%%%%%%%%%%%%%%%%%%%%%%%%%%%%%%%%%%%%%%%%%%%%%%%%%%%%%%%%%%
%References
%%%%%%%%%%%%%%%%%%%%%%%%%%%%%%%%%%%%%%%%%%%%%%%%%%%%%%%%%%%%%%%%%%%%%%%%
\bibliographystyle{abbrv}
\bibliography{vsw_frac}

\end{document}